\documentclass{amsart}
\usepackage{amsmath, amsthm, amsfonts, amssymb}
\usepackage{mathrsfs,graphicx}
\providecommand{\noopsort[1]{}}
\usepackage{bbm}
\usepackage{dsfont}
\usepackage{ifthen}
\numberwithin{equation}{section}
\usepackage{a4}
\usepackage{hyperref}
\usepackage{enumerate}
\usepackage{color}

\newtheorem{thm}{Theorem}[section]
\newtheorem{cor}[thm]{Corollary}
\newtheorem{prop}[thm]{Proposition}
\newtheorem{lem}[thm]{Lemma}

\theoremstyle{remark}
\newtheorem{rem}[thm]{Remark}
\newtheorem{rems}[thm]{Remarks}
\newtheorem{hyp}[thm]{Hypothesis}
\newtheorem{example}[thm]{Example}

\theoremstyle{definition}

\newcommand{\coloneqq}{\mathrel{\mathop:}=}

\newcommand{\dx}{\:\mathrm{d}}
\renewcommand{\Re}{{\rm Re}\,}

\newcommand{\eps}{\varepsilon}
\DeclareMathOperator{\lh}{span}
\DeclareMathOperator{\ran}{ran}
\newcommand{\id}{I}
\newcommand{\one}{\mathds{1}}

\newcommand{\R}{\mathds{R}}

\newcommand{\N}{\mathds{N}}

\newcommand{\fix}{\mathrm{fix}}

\newcommand{\bA}{\mathbf{A}}
\newcommand{\bff}{\mathbf{f}}
\newcommand{\bT}{\mathbf{T}}
\newcommand{\bu}{\mathbf{u}}
\newcommand{\bC}{\mathbf{C}}

\newcommand{\cM}{\mathscr{M}}
\newcommand{\cS}{\mathscr{S}}
\newcommand{\cT}{\mathscr{T}}

\newcommand{\argument}{\,\cdot\,}

\begin{document}
\title[Stability of transition semigroups]{Stability of transition semigroups and applications to parabolic equations}
\author{Moritz Gerlach}
\address{Moritz Gerlach, Universit\"at Potsdam, Institut f\"ur Mathematik, Karl--Liebknecht--Stra{\ss}e 24–25, 14476 Potsdam, Germany}
\email{gerlach@math.uni-potsdam.de}
\author{Jochen Gl\"uck}
\address{Jochen Gl\"uck, Universit\"at Passau, Fakult\"at f\"ur Informatik und Mathematik, Innstra{\ss}e 33, 94032 Passau, Germany}
\email{jochen.glueck@uni-passau.de}
\author{Markus Kunze}
\address{Markus Kunze, Universit\"at Konstanz, Fachbereich Mathematik und Statistik, Fach 193, 78357 Konstanz, Germany}
\email{markus.kunze@uni-konstanz.de}

\begin{abstract}
	The paper deals with the long-term behavior of positive operator semigroups on spaces of bounded functions and of signed measures, which have applications to parabolic equations with unbounded coefficients and to stochastic analysis.
	The main results are a Tauberian type theorem characterizing the convergence to equilibrium of strongly Feller semigroups and a generalization of a classical convergence theorem of Doob. None of these results requires any kind of time regularity of the semigroup.
\end{abstract}

\dedicatory{Dedicated with gratitude to our teacher Wolfgang Arendt}

\maketitle

\section{Introduction}

\subsection*{PDEs, Markov processes and semigroups}
There is a well known connection between parabolic partial differential equations, their (sub)-Markovian solution semigroups on the space of bounded continuous functions, and the associated stochastic process whose transition probability can be recovered from the semigroup -- which is why the latter is often referred to as the \emph{transition semigroup} of the process. This connection allows for an analytical approach to study Markov processes, see \cite{lorenzi2017} and the references therein for more information. 

In contrast to the classical theory of parabolic PDEs on bounded domains, the coefficients appearing in the parabolic equations associated to Markov processes are often unbounded; likewise, the associated semigroups are not strongly continuous, in general. Therefore, weaker time regularity concepts are used, such as pointwise continuity (sometimes referred to as \emph{stochastic continuity}) as was done in \cite{cerrai1994, priola1999}, or continuity with respect to the topology of uniform convergence on compact sets (see e.g.\ \cite{kuehnemund2003}). In some applications even pointwise continuity fails, for instance when considering nonlocal boundary conditions \cite{arendt_kunkel_kunze2016, kunze2020}. A treatment of time continuity properties of Markov semigroups on spaces of measures can be found in \cite{hille_worm2009}.

\smallskip

\subsection*{Convergence results and Doob's theorem}
An important aspect in the study of transition semigroups is the asymptotic behavior as $t \to \infty$.
As the semigroups in question are not strongly continuous, the rich theory for the long-term behavior of $C_0$-semigroups (\cite[Chapter V]{engel_nagel2000}) cannot be used, though. Instead, an important tool is a celebrated result of Doob \cite{doob1948}, later strengthened in \cite{stettner1994, seidler1997}, which gives convergence of stochastically continuous transition semigroups under a number of irreducibility and smoothing assumptions. The theorem is also treated in \cite[Section~4.2]{daprato_zabczyk1996}.

Doob's original proof is of probabilistic nature, as is a recent new proof in \cite{kulik_scheutzow2015}. In \cite{gerlach_nittka2012}, an analytical proof of Doob's theorem was given based on an older theorem of Greiner \cite[Korollar~3.11]{greiner1982}. On the other hand, this theorem of Greiner is the starting point for a series of convergence results for positive semigroups, first for semigroups on sequence spaces \cite{davies2005, keicher2006, wolff2008} and, subsequently, for semigroups of integral operators on function spaces \cite{arendt2008, gerlach2013, gerlach_glueck2017}. The most recent articles in this direction are \cite{gerlach_glueck2019, glueck_haase2019} where all time regularity assumptions on the semigroup were dropped and general semigroup representations on Banach lattices were considered.

\subsection*{Contributions of this article}
In this article, we leave the abstract setting of \cite{gerlach_glueck2019} and come back to actual transition semigroups. For a measurable space $\Omega$ we consider a dual pair of transition semigroups on the space $B_b(\Omega)$ of bounded measurable functions and the space $\mathscr{M}(\Omega)$ of signed measures. From the perspective of stochastic analysis, the semigroup on $B_b(\Omega)$ solves the Kolmogorov backward equation, while the semigroup on $\mathscr{M}(\Omega)$ solves the Kolmogorov forward equation (or Fokker--Planck equation).

The fact that $\mathscr{M}(\Omega)$ is a so-called \emph{AL-space} makes it possible to exploit the results of \cite{gerlach_glueck2019} in this concrete setting to prove the main results of the article: two Tauberian theorems (\ref{t.main-bdd-functions} and~\ref{t.main-cont-functions}), and a generalization of Doob's theorem (Theorem~\ref{t.doob-reloaded}).
As our proofs rely on the approach from \cite{gerlach_glueck2019}, no time regularity assumption is needed, and the results are valid for bounded semigroups of positive kernel operators, rather than only for Markovian semigroups. This is useful, for instance, when studying coupled systems of equations (see Sections~\ref{s.appl-irred-systems} and~\ref{s.appl-red-systems}).

\subsection*{Organization of the article}
Section 2 contains a convergence theorem for positive semigroups on general AL-spaces.
Sections~\ref{s.kernels}--\ref{s.doobs-theorem} focus on transition semigroups over measurable spaces: Section~\ref{s.kernels} contains basics about transition kernels and a duality result for weak convergence on $B_b(\Omega)$ and $\cM(\Omega)$, and Sections~\ref{s.norm-conv-and-tauber} and Section~\ref{s.strong-feller-semigroups} provide Tauberian theorems for transition semigroups. Section~\ref{s.doobs-theorem} gives a generalization of Doob's classical convergence theorem.
To demonstrate the usefulness of our results, we apply them to several classes of parabolic PDEs in Section~\ref{s.applications}.

\subsection*{Acknowledgements}

The authors are grateful to Markus Haase for pointing out to them the argument used in the proof of Proposition~\ref{p.fixed-space-of-contractions} in the appendix. 
The article \cite{denisov2005} that is referred to in Section \ref{s.appl-whole-space} was -- indirectly -- brought to the authors' attention via an answer to a reference request 
on MathOverflow \cite{mathoverflow_stability2019}.

\section{Stability of semigroups on AL-spaces} \label{s.stability-on-al-spaces}

In this section, the asymptotic behavior of bounded positive semigroups on so-called \emph{AL-spaces} is discussed. Recall that an AL-space is a Banach lattice $E$ such that the norm is additive on the positive cone, i.e., for $x,y \in E$ with
$x,y \geq 0$ one has $\|x+y\| = \|x\| + \|y\|$. Every $L^1$-space over an arbitrary measure space is an AL-space.
In the subsequent sections, the most important instance of an AL-space will be the space $\mathscr{M}(\Omega)$ of signed, finite measures on a measurable space $\Omega$. 
Recall that a bounded, positive operator $S$ on a Banach lattice $E$ is called \emph{AM-compact} if it maps order intervals to relatively compact sets.

\begin{thm}\label{t.alconvergence}
	Let $E$ be an AL-space and let $\mathscr{S}= (S_t)_{t\in (0,\infty)}$ be a bounded semigroup of positive operators on $E$. Assume that $S_{t_0}$ is AM-compact for some $t_0>0$ and that $\fix \mathscr{S}$ separates $\fix \mathscr{S}^*$.
	
	Then $S_t$ converges strongly as $t\to \infty$. 
\end{thm}

Here, $\fix \mathscr{S}$ denotes the \emph{fixed space} of $\mathscr{S}$, and $\mathscr{S}^*$ denote the dual semigroup of $\mathscr{S}$ that acts on the norm dual $E^*$ of $E$.
For bounded $C_0$-semigroups, the assumption that $\fix \mathscr{S}$ separates $\fix \mathscr{S^*}$ is equivalent to mean ergodicity of $\mathscr{S}$ \cite[Thm.\ V.4.5]{engel_nagel2000}; thus, Theorem \ref{t.alconvergence} can be interpreted as a Tauberian theorem, as it concludes convergence from (a generalization of) mean ergodicity.

It is worthwhile to mention that if $E$ is an $L^1$-space over a $\sigma$-finite measure space, then every so-called \emph{integral operator} satisfies the assumption of being AM-compact; see for instance \cite[Appendix~A]{gerlach_glueck2019} for details.

The proof of Theorem~\ref{t.alconvergence} at the end of this section consists of two main ingredients: (i) The observation that the theorem is true if $\mathscr{S}$ has sufficiently many fixed points in $E$ (Corollary~\ref{c.gg19}). This is a consequence of a recent result from \cite{gerlach_glueck2019} which is recalled in Theorem~\ref{t.gg19} below. (ii) The observation that the theorem is true if $\mathscr{S}^*$ has no fixed-points at all (Lemma~\ref{l.conv0}). The proof of the latter heavily uses the AL-structure of the space.

As for the first ingredient, we recall the following result from \cite[Thm.\ 3.5]{gerlach_glueck2019}. A vector $x\geq 0$ in a Banach lattice $E$ is called a \emph{quasi-interior point} of the positive cone if the principle ideal generated by $x$ is dense in $E$.

\begin{thm} \label{t.gg19}
	Let $E$ be a Banach lattice and $\mathscr{S} = (S_t)_{t\in (0,\infty)}$ be a bounded positive semigroup on $E$. Assume that the operator $S_{t_0}$ is AM-compact for some $t_0 > 0$ and that the semigroup possesses a fixed point that is a quasi-interior point of the positive cone $E_+$. Then $S_t$ converges strongly as $t\to \infty$.
\end{thm}

If $E$ does not contain quasi-interior points Theorem~\ref{t.gg19} is not directly applicable. For this reason, the following corollary is useful.

\begin{cor}\label{c.gg19}
	Let $E$ be a Banach lattice and $\mathscr{S} = (S_t)_{t\in (0,\infty)}$ be a bounded positive semigroup on $E$. Assume that $S_{t_0}$ is AM-compact for some $t_0>0$ and that the only closed ideal in $E$ that contains all positive fixed points of $\mathscr{S}$ is $E$ itself. Then $S_t$ converges strongly as $t\to \infty$.
\end{cor}

\begin{proof}
	Let $I \subseteq E$ be the (not necessarily closed) ideal generated by all positive fixed points of $\mathscr{S}$ and fix $x\in I$; then there is a fixed point $y\geq 0$ of $\mathscr{S}$ such that $|x|\leq y$. Let $F$ be the closure of the principal ideal generated by $y$. The semigroup $\mathscr{S}$ leaves $F$ invariant and its restriction to that space has a quasi-interior fixed point, namely $y$. As, moreover, $S_{t_0}|_{F}$ is AM-compact, Theorem~\ref{t.gg19} applied to $\mathscr{S}|_F$ yields that the limit of $S_tx$ as $t\to \infty$ exists. 

	Thus, for $x\in I$ one has convergence of $S_tx$ as $t\to \infty$. But as $\mathscr{S}$ is bounded and the closure of $I$ equals $E$ by assumption, strong convergence of $S_t$ as $t\to \infty$ follows from the norm completeness of $E$ and a standard $3\eps$-argument.
\end{proof}

The next lemma is the second ingredient for the proof of Theorem~\ref{t.alconvergence}. Here, the special structure of AL-spaces is used for the first time. 

\begin{lem} \label{l.conv0}
	Let $E$ be an AL-space and $\mathscr{S}$ be a bounded, positive semigroup on $E$. If $\fix\mathscr{S}^* = \{0\}$,  then $S_tx \to 0$ as $t \to \infty$ for all $x\in E$.
\end{lem}

\begin{proof}
	Suppose that $S_t$ does not converge strongly to $0$ as $t\to \infty$. As $\mathscr{S}$ is positive and bounded, there is a vector $x\geq 0$ with $x\neq 0$ and an $\eps >0$ such that $\|S_tx\|\geq \eps$ for all $t>0$. 
	To construct a non-zero fixed point $\varphi$ of the adjoint semigroup $\mathscr{S}^*$, consider a shift-invariant positive functional $m \in (\ell^\infty(0,\infty))^*$ that maps the constant one-function to $1$. Such a functional exists by the Markov--Kakutani fixed point theorem. 
	Moreover, since $E$ is an AL-space, there exists a positive functional $\one \in E^*$ such that $\one (x) = \|x\|$ for all $x\geq 0$. 
	Define $\varphi \in E^*$ by
	\[
		\langle \varphi, y\rangle_{E^*, E} \coloneqq \Big\langle m, \big( \langle \one, S_ty\rangle_{E^*, E}\big)_{t\in (0,\infty)}\Big\rangle_{(\ell^\infty(0,\infty))^*, \ell^\infty(0,\infty)}
	\]
	for each $y \in E$. As $m$ is shift invariant, $\varphi$ is a fixed point of $\mathscr{S}^*$. Also, $\varphi \neq 0$, as
	\[
		\langle \varphi , x\rangle = \big\langle m , (\|S_tx\|)_{t\in (0,\infty)}\big\rangle \geq \langle m, \eps\rangle = \eps >0. \qedhere
	\]
\end{proof}

For $C_0$-semigroups, the assertion of Lemma~\ref{l.conv0} is well-known and can for instance be found in \cite[Thm.\ 7.7]{Chill2007}.
The proof of the following lemma is straightforward, so we omit it.

\begin{lem}
	\label{l.tsc}
	Let $E$ be a Banach space, $F$ a closed subspace of $E$ and $\mathscr{S} = (S_t)_{t\in (0,\infty)}$ a bounded semigroup
	on $E$ such that $S_tF\subseteq F$ for all $t>0$. Denote by $\mathscr{S}^F = (S_t^F)_{t\in (0,\infty)}$ the restriction of the semigroup to $F$ and by $\mathscr{S}^{E/F} = (S_t^{E/F})_{t\in (0,\infty)}$ the  quotient semigroup. Then the following are equivalent:
	\begin{enumerate}[\upshape (i)]
		\item $S_t$ is strongly convergent as $t \to \infty$ and the limit operator maps into $F$.
		
		\item $S_t^{E/F}$ converges strongly to zero and $S_t^F$ is strongly convergent as $t \to \infty$.
	\end{enumerate}
\end{lem}

One further observation is still missing for the proof of Theorem~\ref{t.alconvergence}.

\begin{lem} \label{l.ex-of-positive-fixed-point}
	Let $E$ be an AL-space and $\cS$ be a bounded, positive semigroup on $E$. For every $x \in \fix \cS$ there exists a positive vector $y \in \fix \cS$ such that $|x| \leq y$.
\end{lem}
\begin{proof}
	One has $S_t|x|\geq |S_tx| = |x|$ for all $t>0$ and thus $S_{t+s}|x| \geq S_s|x|$ for all $t,s>0$. 
	So the norm bounded net
	$(S_t|x|)_{t\geq 0}$ is increasing and hence a Cauchy net as $E$ is an AL-space.
	Hence, $(S_t|x|)_{t\geq 0}$ converges to some $y\in \fix \cS$, and $y \ge |x|$.
\end{proof}

\begin{proof}[Proof of Theorem~\ref{t.alconvergence}]
	Let $J$ be the closed ideal generated by $\fix \mathscr{S}$. By Lemma~\ref{l.ex-of-positive-fixed-point}, $J$ is generated by $\fix \mathscr{S} \cap E_+$. An application of Corollary~\ref{c.gg19} to the restriction $\mathscr{S}^J$ of $\mathscr{S}$ to $J$ yields the strong convergence of this semigroup.

	We claim that the dual of the quotient semigroup $\mathscr{S}^{E/J}$ has trivial fixed space. To see this, let 
$q: E \to E/J$ be the quotient map and $\varphi \in (E/J)^*$ be a fixed point of $\big(\mathscr{S}^{E/J}\big)^*$. Then
$\varphi \circ q$ is a fixed point of $\mathscr{S}^*$ that vanishes on $\fix\mathscr{S}\subseteq J$. By assumption,
$\varphi\circ q =0$. But $q$ is surjective, so $\varphi =0$.

	So Lemma~\ref{l.conv0} yields strong convergence to $0$ of $\mathscr{S}^{E/J}$. By Lemma~\ref{l.tsc}, $\mathscr{S}$ converges strongly as claimed.
\end{proof}

\section{Kernels and kernel operators} \label{s.kernels}

In this section, we consider semigroups of so-called \emph{kernel operators} on a space of measurable functions, and their dual semigroups on the space of measures. The relation of weak convergence of the two semigroups is discussed and a description of the limit operator is given. This forms a very useful frame of reference for what follows in the subsequent sections.

Let $(\Omega, \Sigma)$ be a measurable space and let $\mathscr{M}(\Omega)$ and $B_b(\Omega)$ denote, respectively, the space of signed (finite) measures on $\Omega$ and the space of bounded real-valued measurable functions on $\Omega$. As usual, $\mathscr{M}(\Omega)$ is endowed with the total variation norm and $B_b(\Omega)$ with the supremum norm, which renders both spaces Banach lattices.

A \emph{bounded kernel} on $\Omega$ is a mapping $k: \Omega\times \Sigma \to \R$ such that
\begin{enumerate}[(a)]
	\item for every $A\in \Sigma$ the map $x\mapsto k(x,A)$ is measurable,
	\item for every $x \in \Omega$ the map $A\mapsto k(x, A)$ is a (signed) measure,
	\item $\sup_{x\in \Omega} |k|(x, \Omega) < \infty$, where $|k|(x,\cdot)$ denotes the total variation of the measure $k(x,\cdot)$.
\end{enumerate}

To each bounded kernel $k$ there corresponds a bounded linear operator $S: \mathscr{M}(\Omega) \to \mathscr{M}(\Omega)$, defined via
\begin{align}
	\label{eq:kernel-operator-on-measures}
	(S\mu)(A) = \int_\Omega k(x,A) \dx\mu(x) \quad \text{for }  A \in \Sigma,
\end{align}
for all $\mu \in \cM(\Omega)$ and a bounded linear operator $T: B_b(\Omega) \to B_b(\Omega)$, defined by
\begin{align}
	\label{eq:kernel-operator-on-functions}
	(Tf)(x) = \int_\Omega f(y) \, k(x,\dx y) \quad \text{for } x \in \Omega
\end{align}
for all $f\in B_b(\Omega)$. It is not difficult to see that $\|S\| = \|T\| = \sup_{x\in \Omega} |k|(x, \Omega)$, and that the operators $S$ and $T$ are in duality in the sense that $\langle S \mu, f\rangle = \langle \mu, Tf \rangle$ for each measure $\mu \in \mathscr{M}(\Omega)$ and each function $f \in B_b(\Omega)$. Here, the notation $\langle \mu, f\rangle := \int_\Omega f(x) \dx \mu(x)$ is used.

Now consider the situation the other way round. First, let $S: \mathscr{M}(\Omega) \to \mathscr{M}(\Omega)$ be a bounded linear operator. It follows from \cite[Prop.\ 3.1 and~3.5]{kunze2011} that the following properties are equivalent:
\begin{enumerate}[(i)]
	\item There exists a bounded kernel $k$ on $\Omega$ such that $S$ is given by the integral formula~\eqref{eq:kernel-operator-on-measures}.
		
	\item The norm adjoint $S^*$ leaves $B_b(\Omega)$ invariant.
	
	\item The operator $S$ is continuous with respect to the $\sigma(\mathscr{M}(\Omega), B_b(\Omega))$-topology.
\end{enumerate}

If $S$ satisfies these equivalent conditions, then $S$ is called a \emph{kernel operator} with associated kernel $k$. In this case, the kernel $k$ is uniquely determined and the kernel operator $S$ is positive (in which case it is called a \emph{positive kernel operator}) if and only if $k$ is positive in the sense that $k(x,A) \ge 0$ for all $x \in \Omega$ and all measurable $A \subseteq \Omega$. If $S$ is a kernel operator with associated kernel $k$, then the restriction of its adjoint to $B_b(\Omega)$ is denoted by $S'$. Note that $S'$ is given by the integral formula~\eqref{eq:kernel-operator-on-functions}.

Analogously, one can start with a bounded linear operator $T: B_b(\Omega) \to B_b(\Omega)$. Again, it follows from \cite[Prop.\ 3.1 and~3.5]{kunze2011} that the following properties are equivalent:
\begin{enumerate}[(i)]
	\item There exists a bounded kernel $k$ on $\Omega$ such that $T$ is given by the integral formula~\eqref{eq:kernel-operator-on-functions}.
		
	\item The norm adjoint $T^*$ leaves $\mathscr{M}(\Omega)$ invariant.
	
	\item The operator $T$ is continuous with respect to the $\sigma(B_b(\Omega), \mathscr{M}(\Omega))$-topology.
\end{enumerate}

If $T$ satisfies these equivalent conditions, then it is called $T$ a \emph{kernel operator} with kernel $k$; as above, the kernel $k$ is uniquely determined by $T$ and positivity of $T$ is equivalent to positivity of $k$. The restriction of $T^*$ to $\mathscr{M}(\Omega)$ is denoted by $T'$, and it is given by the integral formula~\eqref{eq:kernel-operator-on-measures}.

In other words, if $T$ is a kernel operator on $B_b(\Omega)$ with associated kernel $k$, then $T'$ is a kernel operator on $\mathscr{M}(\Omega)$ with the same kernel; and if $S$ is a kernel operator on $\mathscr{M}(\Omega)$ with associated kernel $k$, then $S'$ is a kernel operator on $B_b(\Omega)$ with the same kernel.

Convergence of kernel operators is related to pointwise convergence of the associated kernels in the sense of the following proposition, which can easily be derived from the Vitali--Hahn--Saks theorem \cite[Thm.\ 4.6.3]{bogachev2007}.

\begin{prop} 
	\label{p.convergence-of-kernel-ops}
	Let $(\Omega, \Sigma)$ be a measurable space and let $(k_n)$ be a sequence of bounded kernels on $\Omega$ which is bounded in the sense that
	\begin{align*}
		\sup_{n \in \N} \sup_{x \in \Omega} |k_n|(x,\Omega) < \infty.
	\end{align*}
	Denote the corresponding kernel operators on $\mathscr{M}(\Omega)$ and $B_b(\Omega)$ by $(S_n)$ and $(T_n)$, respectively. The following are equivalent:
	\begin{enumerate}[\upshape (i)]
		\item For each $x \in \Omega$ and each $A \in \Sigma$, the sequence $(k_n(x,A))$ converges to a real number $k(x,A)$.
		
		\item For each $\mu \in \cM(\Omega)$ the sequence $(S_n\mu)$ is $\sigma(\cM(\Omega), B_b(\Omega))$-convergent to a measure $S\mu \in \cM(\Omega)$.
		
		\item For each $f \in B_b(\Omega)$ the sequence $(T_nf)$ is $\sigma(B_b(\Omega), \cM(\Omega))$-convergent to a function $Tf \in B_b(\Omega)$.
	\end{enumerate}
	If these equivalent assertions are satisfied, then $k$ is a bounded kernel on $\Omega$, $S$ is a kernel operator on $\cM(\Omega)$ with kernel $k$, and $T$ is a kernel operator on $B_b(\Omega)$ with kernel $k$.
\end{prop}

Note that the boundedness condition on the sequence $(k_n)$ is equivalent to either of the sequences $(S_n)$ or $(T_n)$ being norm bounded.

It is worthwhile to explicitly formulate the equivalence of assertions~(ii) and~(iii) above for the special case of semigroups. In this case, one also obtains additional information about the structure of the limit operator, similarly to mean ergodic semigroups, see \cite{gerlach_kunze2014}.
We use the abbreviations $\ran(\id - \cS)$ and $\ran(\id - \cT)$ for the spaces
\begin{align*}
	\lh \biggl(\bigcup_{t \in (0,\infty)} (\id - S_t)\cM(\Omega)\biggr) \qquad \text{and} \qquad \lh \biggl(\bigcup_{t \in (0,\infty)} (\id - T_t)B_b(\Omega)\biggr),
\end{align*}
respectively.

\begin{prop} 
	\label{p.weak-convergence-of-sg}
	Let $(\Omega, \Sigma)$ be a measurable space, let $\cT = (T_t)_{t \in (0,\infty)}$ be a bounded semigroup of kernel operators on $B_b(\Omega)$ and  denote the dual semigroup on $\cM(\Omega)$ by $\cS = (S_t)_{t \in (0,\infty)} := (T_t')_{t \in (0,\infty)}$. The following are equivalent:
	
	\begin{enumerate}[\upshape (i)]
		\item 
		For each $\mu \in \cM(\Omega)$, $S_t\mu$ is $\sigma(\cM(\Omega), B_b(\Omega))$-convergent to a measure $P\mu \in \cM(\Omega)$ as $t \to \infty$.
		
		\item 
		For each $f \in B_b(\Omega)$, $T_tf$ is $\sigma(B_b(\Omega), \cM(\Omega))$-convergent to a function $Qf \in B_b(\Omega)$ as $t \to \infty$.
	\end{enumerate}
	
	In this case $P$ is a kernel operator on $\cM(\Omega)$ that commutes with each operator $S_t$, and $Q$ is a kernel operator on $B_b(\Omega)$ that commutes with each operator $T_t$. Moreover, $P$ and $Q$ are in duality, and they have the following properties:
	
	\begin{enumerate}[\upshape (a)]
		\item The space $\cM(\Omega)$ splits as $\cM(\Omega) = \fix \cS \oplus \overline{\ran(\id - \cS)}$ (where the closure is taken with respect to $\sigma(\cM(\Omega), B_b(\Omega))$) and $P$ is the projection onto $\fix \cS$ along this splitting.
		
		\item The space $B_b(\Omega)$ splits as $B_b(\Omega) = \fix \cT \oplus \overline{\ran(\id - \cT)}$ (where the closure is taken with respect to $\sigma(B_b(\Omega), \cM(\Omega))$) and $Q$ is the projection onto $\fix \cT$ along this splitting.
	\end{enumerate}
	
	Finally, the spaces $\fix \cT$ and $\fix \cS$ have the same dimension.
\end{prop}

This is a straightforward consequence of Proposition~\ref{p.convergence-of-kernel-ops}, so we omit the details of the proof.

\section{Convergence in total variation norm and a Tauberian theorem} \label{s.norm-conv-and-tauber}

The main point of Proposition~\ref{p.weak-convergence-of-sg} was to relate, for a semigroup of kernel operators, weak convergence on the space of measurable functions to weak convergence on the space of measures. For practical purposes, two additional properties are desirable: (i) on the space of measures, one is not only interested in weak convergence, but rather in convergence in total variation norm; (ii) it is important to have sufficient -- or equivalent -- conditions for convergence that are easy to check in concrete situations.

As it turns out, both requirements are met by the following theorem. Recall that a $\sigma$-algebra $\Sigma$  is called \emph{countably generated} if there exists a countable family of sets $\mathscr{A} \subseteq \Sigma$ such that $\Sigma$ is the smallest $\sigma$-algebra that contains $\mathscr{A}$. For instance, the Borel $\sigma$-algebra on a Polish space is countably generated.

\begin{thm} \label{t.main-bdd-functions}
	Let $(\Omega, \Sigma)$ be a measurable space whose $\sigma$-algebra is countably generated. Let $\cT = (T_t)_{t \in (0,\infty)}$ be a bounded semigroup of positive kernel operators on $B_b(\Omega)$ with associated kernels $k_t$, and let $\cS = (S_t)_{t \in (0,\infty)} := (T_t')_{t \in (0,\infty)}$ denote the dual semigroup on $\cM(\Omega)$. 
	
	Suppose that there exists a measure $0 \le \mu \in \cM(\Omega)$ and a time $t_0 \in (0,\infty)$ such that, for each $x \in \Omega$, the measure $k_{t_0}(x,\argument)$ is absolutely continuous with respect to $\mu$. Then the following assertions are equivalent:
	\begin{enumerate}[\upshape (i)]
		\item For each $f \in B_b(\Omega)$, $T_tf$ converges with respect to $\sigma(B_b(\Omega),\cM(\Omega))$ to a function in $B_b(\Omega)$ as $t \to \infty$.
		
		\item For each $\nu \in \cM(\Omega)$, $S_t\nu$ converges in total variation norm to a measure in $\cM(\Omega)$ as $t \to \infty$.
		
		\item The space $\fix \cS$ separates $\fix \cT$.
	\end{enumerate}
\end{thm}

Similar to Theorem \ref{t.alconvergence}, Theorem \ref{t.main-bdd-functions} can be considered as a Tauberian theorem; indeed, in a slightly different context it was proved in \cite[Thm.\ 5.7]{gerlach_kunze2014}, that condition (iii) is equivalent with (weak) mean ergodicity. 
It is worthwhile to point out that, under the assumptions of the theorem, $\fix \cT$ always separates $\fix \cS$; see Remark~\ref{r.converse-separation} below. 

Since $\cM(\Omega)$ is an AL-space the results of Section~\ref{s.stability-on-al-spaces} can be used to prove Theorem~\ref{t.main-bdd-functions}. The following auxiliary result is helpful.

\begin{lem}\label{l.amcompact}
	Let $(\Omega, \Sigma)$ be a measurable space, let $S$ be a positive kernel operator on $\mathscr{M}(\Omega)$ and denote the associated kernel by $k$. Assume that there exists a finite positive measure $0 \le \mu \in \mathscr{M}(\Omega)$ such that, for every $x\in \Omega$, the measure $k(x,\cdot)$ is absolutely continuous with respect to $\mu$. Then the following assertions hold:
	\begin{enumerate}
		\item[\upshape (a)] The range of $S$ is contained in the band $\{\mu\}^{\perp\perp}$ generated by $\mu$ in $\mathscr{M}(\Omega)$.
		
		\item[\upshape (b)] The norm adjoint $S^*$ of $S$ maps the norm dual $\cM(\Omega)^*$ into $B_b(\Omega)$.
	\end{enumerate}
	Now assume in addition that the $\sigma$-algebra $\Sigma$ is countably generated. Then:
	\begin{enumerate}
		\item[\upshape (c)] The restriction $S|_{\{\mu\}^{\perp\perp}}$ is an AM-compact operator on $\{\mu\}^{\perp\perp}$.
		
		\item[\upshape (d)] The operator $S^2$ is AM-compact on $\mathscr{M}(\Omega)$.
	\end{enumerate}
\end{lem}
\begin{proof}
	(a) First observe that the band $\{\mu\}^{\bot\bot}$ consists exactly of those measures that are absolutely continuous with respect to $\mu$, cf.\ \cite[Thm.\ 10.61]{aliprantis_border2006}. Thus, it follows from Formula~\eqref{eq:kernel-operator-on-measures} that $S$ maps $\mathscr{M}(\Omega)$ into $\{\mu\}^{\perp\perp}$. 
	
	(b) The band $\{\mu\}^{\bot\bot}$ is isometrically Banach lattice isomorphic to $L^1(\Omega,\mu)$ via the mapping $\Phi\colon \{\mu\}^{\bot\bot} \to L^1(\Omega,\mu)$ which maps each $\nu \in \{\mu\}^{\bot\bot}$ to its Radon-Nikodym derivative with respect to $\mu$, see \cite[Thm.\ 13.19]{aliprantis_border2006}.
	
	Let $\varphi \in \cM(\Omega)^*$. Its restriction to $\{\mu\}^{\perp\perp}$ induces, via $\Phi$, a functional on $L^1(\Omega,\mu)$ 
	which can in turn be represented by a function $\tilde f \in L^\infty(\Omega,\mu)$. By choosing a representative $f \in B_b(\Omega)$ of $\tilde f$, one obtains for each $\nu \in \cM(\Omega)$
	\begin{align*}
		& \langle S^*\varphi, \nu \rangle_{\langle \cM(\Omega)^*, \cM(\Omega)\rangle} = \langle \varphi|_{\{\mu\}^{\perp\perp}}, S\nu \rangle_{\langle (\{\mu\}^{\perp\perp})^*, \{\mu\}^{\perp\perp} \rangle} \\
		& = \langle \tilde f, \Phi S\nu \rangle_{\langle L^\infty(\Omega,\mu), L^1(\Omega,\mu) \rangle} = \langle f, S\nu \rangle_{\langle B_b(\Omega), \cM(\Omega) \rangle} = \langle S'f, \nu \rangle_{\langle B_b(\Omega), \cM(\Omega) \rangle}.
	\end{align*}
	Hence, $S^*\varphi = S'f \in B_b(\Omega)$.
	
	For the rest of the proof, let $\Sigma$ be countably generated.
	
	(c) As in the proof of (b) the band $\{\mu\}^{\bot\bot}$ can be identified with $L^1(\Omega, \mu)$; to keep the argument concise,
	the isomorphism $\Phi$, that was explicitly mentioned in the proof of~(b), is now suppressed in the notation.
	
	It thus suffices to prove that $S|_{L^1(\Omega, \mu)}$ is AM-compact.
	By assumption, for every $x\in \Omega$ there is $k_x\in L^1(\Omega,\mu)$ such that $k(x, \cdot) = k_xd\mu$. Then the mapping
	$K \colon \Omega \to L^1(\Omega, \mu)$, defined by $Kx = k_x$ is weakly measurable. As $\Sigma$ is countably generated, 
	$L^1(\Omega, \mu)$ is separable, whence $K$ is also separably valued. The Pettis measurability theorem yields that $K$ is strongly measurable. Since $\|k_x\|_{L^1} = k(x, \Omega)$ is integrable
	by assumption, $K \in L^1(\Omega, \mu; L^1(\Omega, \mu)) \simeq L^1(\Omega\times\Omega, \mu\otimes\mu)$. 
	By identifying $K$ with $h\in L^1(\Omega\times\Omega, \mu\otimes\mu)$, one obtains
	\[
		\big[Sf\big](x) = \int_\Omega h(x,y)\, d\mu(y),
	\]
	i.e.\ $S|_{L^1(\Omega, \mu)}$ is a so-called \emph{integral operator}. Since integral operators are AM-compact	(see for instance \cite[Theorem~A.3]{glueck_haase2019} or \cite[Appendix~A]{gerlach_glueck2019}), this implies the assertion.
	
	(d) This is an immediate consequence of~(a) and~(c).
\end{proof}

\begin{proof}[Proof of Theorem~\ref{t.main-bdd-functions}]
	``(ii) $\Rightarrow$ (i)'' Clearly, convergence in total variation norm implies convergence with respect to the $\sigma(\cM(\Omega), B_b(\Omega))$-topology. Hence, this implication follows from Proposition~\ref{p.weak-convergence-of-sg}.
	
	``(i) $\Rightarrow$ (iii)'' Assume that~(i) holds, and let $Q: B_b(\Omega) \to B_b(\Omega)$ denote the limit operator. Proposition~\ref{p.weak-convergence-of-sg} implies that $Q$ is a kernel operator and a projection onto $\fix \cT$, and that $P := Q': \cM(\Omega) \to \cM(\Omega)$ is a projection onto $\fix \cS$.
	
	Now, let $f \in B_b(\Omega)$ be a non-zero fixed vector of $\cT$. Choose a measure $\nu \in \cM(\Omega)$ that does not vanish on $f$ and observe that
	\begin{align*}
		0 \not= \langle \nu, f \rangle = \langle \nu, Qf\rangle = \langle P\nu, f\rangle.
	\end{align*}
	Hence, $\fix \cS$ separates $\fix \cT$.
	
	``(iii) $\Rightarrow$ (ii)'' According to Lemma~\ref{l.amcompact}(b), the norm adjoint of $S_{t_0}$ maps $\cM(\Omega)^*$ into $B_b(\Omega)$; every fixed point $\varphi$ of $\cS^*$ thus satisfies $\varphi = S_{t_0}^*\varphi \in B_b(\Omega)$. Hence, $\fix \cS^* = \fix \cT$ and, consequently, $\fix \cS$ separates $\fix \cS^*$ by~(iii).
	
	Since $\cS$ contains an AM-compact operator by Lemma~\ref{l.amcompact}(d), the assertion follows from the convergence result in Theorem~\ref{t.alconvergence}.
\end{proof}

\begin{rem} \label{r.converse-separation}
	Under the assumptions of  Theorem~\ref{t.main-bdd-functions}, $\fix \cT$  always separates $\fix \cS$,  no matter whether the equivalent assertions~(i)--(iii) are satisfied or not.
	
	Indeed, let $\nu \in \fix \cS$. Then the norm adjoint semigroup $\cS^*$ on the norm dual space $\cM(\Omega)^*$ has a fixed point $\varphi$ such that $\langle \varphi, \nu \rangle \not= 0$. To see this, choose a functional $\psi \in \cM(\Omega)^*$ such that $\langle \psi, \nu \rangle \not= 0$ and let $b \in \big(\ell^\infty(0,\infty)\big)^*$ be a Banach limit (which exists by the Markov--Kakutani fixed point theorem). Define $\varphi \in \cM(\Omega)^*$ by means of the formula
	\begin{align*}
		\langle \varphi, \lambda \rangle = \left \langle b \, , \, \big(\langle \psi, S_t \lambda \rangle\big)_{t \in (0,\infty)} \right \rangle
	\end{align*}
	for each $\lambda \in\cM(\Omega)$. As $b$ is a Banach limit, it follows that $\varphi$ is a fixed point of $\cS^*$; moreover,  $\langle \varphi, \nu \rangle = \langle \psi, \nu \rangle \not= 0$, so $\varphi$ does not vanish on $\nu$. 
	
	Finally, Lemma~\ref{l.amcompact} implies that $S_{t_0}^*$ maps $\cM(\Omega)^*$ into $B_b(\Omega)$. Hence, $\varphi = S_{t_0}^* \varphi \in B_b(\Omega)$, i.e., $\varphi$ is a fixed point of $\cT$ that does not vanish on the given fixed point $\nu$ of $\cS$.
\end{rem}

\section{Stability of strongly Feller semigroups} \label{s.strong-feller-semigroups}

This section focusses on the special case where $\Omega$ is a \emph{Polish space}, i.e., a separable topological space that is metrizable through a complete metric, endowed with its Borel $\sigma$-algebra $\mathscr{B}(\Omega)$.
Besides $B_b(\Omega)$ we shall also consider the space $C_b(\Omega)$ of bounded continuous functions on $\Omega$. When endowed with the supremum norm, this is a Banach lattice and in fact a  closed sublattice of $B_b(\Omega)$.

A kernel operator $T$ on $B_b(\Omega)$ is called \emph{strongly Feller} if $TB_b(\Omega) \subseteq C_b(\Omega)$. Operators that are strongly Feller have the following very useful property, which makes it possible to apply the results from Section~\ref{s.norm-conv-and-tauber} -- in particular, Theorem~\ref{t.main-bdd-functions} -- to semigroups that contain a strongly Feller operator.

\begin{prop} \label{p.dual-of-strong-feller-op}
	Let $\Omega$ be a Polish space and let $T$ be a positive kernel operator on $B_b(\Omega)$ with kernel $k$. If $T$ is strongly Feller, then there exists a measure $0 \le \mu \in \mathscr{M}(\Omega)$ such that $k(x,\argument)$ is absolutely continuous with respect to $\mu$ for each $x \in \Omega$.
\end{prop}
\begin{proof}
	Fixing  a dense sequence $(x_n)_{n\in \N}$ in $\Omega$, it is easy to see that 
	\begin{align*}
		\mu \coloneqq \sum_{n \in \N} 2^{-n} k(x_n,\argument);
	\end{align*}
has the desired properties, see \cite[proof of Lem.\ 1.5.9]{revuz1984}	
\end{proof}

Proposition~\ref{p.dual-of-strong-feller-op} shows that if $T$ on $B_b(\Omega)$ is strongly Feller, then the dual operator $T'$ on $\mathscr{M}(\Omega)$ satisfies the assumptions of Lemma~\ref{l.amcompact}.

There is an even  stronger notion than that of a strongly Feller operator. To discuss this, we use the concept of the \emph{strict topology} $\beta_0$ on $C_b(\Omega)$ for a Polish space $\Omega$, defined as follows. Denote by $\mathscr{F}_0(\Omega)$ the set of all functions $\varphi: \Omega \to \R$ that \emph{vanish at infinity}, i.e.\ given $\eps>0$ there is a compact set $K\subseteq \Omega$ such that $|f(x)| \leq \eps$ for all $x \in \Omega \setminus K$. The \emph{strict topology $\beta_0$} is the locally convex topology generated by the set of seminorms $\{p_\varphi : \varphi \in \mathscr{F}_0(\Omega)\}$, where $p_\varphi(f) \coloneqq \|\varphi f\|_\infty$.

The strict topology has the following properties. It is consistent with the duality $(C_b(\Omega), \mathscr{M}(\Omega))$, i.e., the dual space $(C_b(\Omega), \beta_0)'$ is $\mathscr{M}(\Omega)$, see \cite[Thm.\ 7.6.3]{jarchow1981}. As a matter of fact, it is actually the Mackey topology of the dual pair $(C_b(\Omega), \mathscr{M}(\Omega))$, i.e.\ the finest locally convex topology on $C_b(\Omega)$ which yields $\mathscr{M}(\Omega)$ as a dual space, see \cite[Thm.\ 4.5 and 5.8]{sentilles1972}. It follows that if $T$ is a kernel operator on $B_b(\Omega)$ that leaves $C_b(\Omega)$ invariant, then the restriction of $T$ to $C_b(\Omega)$ is automatically $\beta_0$-continuous. By \cite[Thm.\ 2.10.4]{jarchow1981} $\beta_0$ coincides on $\|\cdot\|_\infty$-bounded subsets of $C_b(\Omega)$ with the compact-open topology. In particular, a $\|\cdot\|_\infty$-bounded net converges with respect to $\beta_0$ if and only if it converges uniformly on every compact subset of $\Omega$.

For a Polish space $\Omega$, a kernel operator $T$ on $B_b(\Omega)$ is called \emph{ultra Feller} if it maps bounded subsets of $B_b(\Omega)$ to relatively $\beta_0$-compact subsets of $C_b(\Omega)$. Clearly, every ultra Feller operator is strongly Feller. On a more interesting note, the following result holds in the converse direction:

\begin{prop}\label{p.ultrafeller}
	Let $\Omega$ be a Polish space. If a positive kernel operator $T$ on $B_b(\Omega)$ is strongly Feller, then $T^2$ is ultra Feller.
\end{prop}

This proposition is quite well-known; for the convenience of the reader we include a proof.

\begin{proof}[Proof of Proposition~\ref{p.ultrafeller}]
	It follows from \cite[Lem.\ 5.10 and~5.11 in Chapter~1]{revuz1984} that if $T$ is a positive, strongly Feller operator, then for any bounded sequence $(f_n)_{n\in \N} \subseteq B_b(\Omega)$ one can extract a subsequence $f_{n_k}$ such that $T^2f_{n_k}$ converges to some continuous function $g$ uniformly on every compact subset of $\Omega$. Since compactness in $C(K)$ is the same as sequential compactness, the classical Arzel\`{a}--Ascoli theorem yields
that, for a bounded subset $B \subseteq B_b(\Omega)$, the image $T^2B \subseteq C_b(\Omega)$ is equicontinuous on any compact subset of $\Omega$. Employing a Arzel\`{a}--Ascoli theorem for the strict topology, see \cite[Thm.\ 3.6]{khan1979}, it follows that $T^2B$ is relatively $\beta_0$-compact.
\end{proof}

\begin{rem}
	Let $\Omega$ be a Polish space and let $T$ be a kernel operator on $B_b(\Omega)$ with dual operator $S$ on $\cM(\Omega)$. If $T$ is strongly Feller, then it follows from Proposition~\ref{p.dual-of-strong-feller-op} and Lemma~\ref{l.amcompact}(b) that the norm adjoint $(S^2)^*$ of $S^2$ maps the norm dual $\cM(\Omega)^*$ of $\cM(\Omega)$ to $C_b(\Omega)$. On the other hand, $S^2$ is the dual kernel operator of the ultra Feller operator $T^2$ (Proposition~\ref{p.ultrafeller}).

	It is worthwhile to mention that a similar result remains true for general ultra Feller operators -- not only for those that are the square of a strongly Feller operator. This is not needed in what follows, but it is an interesting observation for its on sake, so we provide a proof for this fact in the appendix (Proposition~\ref{p.ultrafeller2}).
\end{rem}

The following theorem is the main result in this section. It is an analogue of Theorem~\ref{t.main-bdd-functions} for semigroups that contain a strongly Feller operator.

\begin{thm} \label{t.main-cont-functions}
	Let $\Omega$ be a Polish space, let $\cT = (T_t)_{t \in (0,\infty)}$ be a bounded semigroup of positive kernel operators on $B_b(\Omega)$ and let $\cS = (S_t)_{t \in (0,\infty)} := (T_t')_{t \in (0,\infty)}$ denote the dual semigroup on $\cM(\Omega)$. If $T_{t_0}$ is strongly Feller for some $t_0 \in (0,\infty)$, then the following assertions are equivalent:
	\begin{enumerate}[\upshape (i)]
		\item 
		For each $f\in B_b(\Omega)$, $T_tf$ converges pointwise to a function $Qf\in B_b(\Omega)$ as $t \to \infty$.

		\item 
		For each $f \in B_b(\Omega)$, $T_tf$ converges uniformly on compact subsets of $\Omega$ to a function $Qf$ in $C_b(\Omega)$ as $t \to \infty$.
		
		\item 
		For each $\nu \in \mathscr{M}(\Omega)$, $S_t\nu$ converges with respect to the total variation norm as $t \to \infty$.
		
		\item 
		The space $\fix \cS$ separates $\fix \cT$.
	\end{enumerate}
	If the equivalent conditions~{\upshape(i)}--{\upshape(iv)}~are satisfied, then the limit operator $Q$ in~{\upshape(i)} and {\upshape(ii)} is a kernel operator on $B_b(\Omega)$ that is ultra Feller.
\end{thm}

\begin{proof}
	``(i)~$\Rightarrow$~(ii)'' For $t \ge t_0$ the function $T_tf$ is contained in $C_b(\Omega)$, and for $t \ge 2t_0$, $T_tf$ is -- due to the boundedness of $\cT$ -- contained in a $t$-independent $\beta_0$-compact set $K \subseteq C_b(\Omega)$. Hence, every subnet of $(T_tf)$ has a subnet which is $\beta_0$-convergent, and the limit equals $Qf$ since the $\beta_0$-convergence is stronger than pointwise convergence. Hence, $T_tf$ is $\beta_0$-convergent to $Qf$ as $t\to \infty$. As $\cT$ is bounded, this is equivalent to uniform convergence of compact subsets of $\Omega$. Consequently, the restriction of $Qf$ to any compact subset of $\Omega$ is continuous, which implies -- due to the metrizability of $\Omega$ -- that $Qf$ itself is continuous, i.e., $Qf \in C_b(\Omega)$. This proves~(ii). 
	
	``(ii)~$\Rightarrow$~(i)'' This implication is obvious.
	
	``(i)~$\Rightarrow$~(iii)~$\Leftrightarrow$~(iv)'' These implications follow from Theorem~\ref{t.main-bdd-functions}. To see this, one only has to note that the kernel associated to $T_{t_0}$ satisfies the assumptions of that theorem according to Proposition~\ref{p.dual-of-strong-feller-op}, and that pointwise convergence of a bounded sequence in $B_b(\Omega)$ is equivalent to convergence with respect to the $\sigma(B_b(\Omega), \cM(\Omega))$-topology by the dominated convergence theorem.
	
	``(iii) $\Rightarrow$ (i)'' Since convergence with respect to the total variation norm implies convergence with respect to the $\sigma(\cM(\Omega), B_b(\Omega))$-topology, this implication follows from Proposition~\ref{p.weak-convergence-of-sg}.
	
	Finally, assume that~(i)--(iv) hold. By Theorem~\ref{t.main-bdd-functions} $Q$ is a projection and a kernel operator on $B_b(\Omega)$. Moreover, it follows from~(ii) that $Q$ is strongly Feller, so $Q = Q^2$ is even ultra Feller by Proposition~\ref{p.ultrafeller}.
\end{proof}

\begin{rem}
	In applications it is often more practical to work with a semigroup on $C_b(\Omega)$ rather than with a semigroup on $B_b(\Omega)$. An operator $T$ on $C_b(\Omega)$ is called a \emph{kernel operator} if there exists a kernel $k$ such that
	\eqref{eq:kernel-operator-on-functions} holds for all $f \in C_b(\Omega)$. Note that in this case, $T$ can be extended to a kernel operator on $B_b(\Omega)$ by using the representation \eqref{eq:kernel-operator-on-functions} for $f\in B_b(\Omega)$.
	As $C_b(\Omega)$ is $\sigma(B_b(\Omega), \mathscr{M}(\Omega))$-dense in $B_b(\Omega)$, this extension is unique.
	
	Thus, if $\mathscr{T}_C$ is a semigroup on $C_b(\Omega)$ consisting of kernel operators, then it can be uniquely extended to a semigroup $\mathscr{T}$ on $B_b(\Omega)$. Moreover, if
$\mathscr{T}$ contains a strongly Feller operator, then the fixed points of $\mathscr{T}$ are necessarily continuous
	so that $\fix \mathscr{T} = \fix\mathscr{T}_C$ and knowledge of $\mathscr{T}_C$ is sufficient to verify condition (iv) in Theorem \ref{t.main-cont-functions}. It is not difficult to see that a bounded operator on $C_b(\Omega)$ is a kernel operator if and only if it is $\sigma(C_b(\Omega), \mathscr{M}(\Omega))$-continuous, see \cite[Prop.\ 3.5]{kunze2009}.
\end{rem}

Up to now, no regularity assumptions on the orbits of the semigroup $\mathscr{T}$ have been made. However, semigroups appearing in applications typically have orbits that are regular in some sense, so that a generator can be associated to them. In this situation, it is possible to reformulate condition~(iv) in Theorem \ref{t.main-cont-functions} in terms of the generator, since the (dual) fixed space of the semigroup is then simply the kernel of the (dual) generator. This facilitates the applicability to concrete differential equations.

A generator can be defined under very weak time regularity assumptions, namely a weak measurability assumption.
Consider a bounded semigroup of kernel operators (equivalently, $\sigma(C_b(\Omega), \cM(\Omega))$-continuous operators) $\mathscr{T}_C= (T_t^C)_{t>0}$ on $C_b(\Omega)$; we will explain in a moment why we prefer to work on $C_b(\Omega)$. The semigroup $\mathscr{T}_C$ is called \emph{weakly measurable} if for every $f\in C_b(\Omega)$ and $\mu \in \mathscr{M}(\Omega)$ the scalar function $t\mapsto \langle T_t^Cf, \mu\rangle$ is measurable. It follows from \cite[Thm.\ 6.2]{kunze2011} that $\mathscr{T}_C$ is an integrable semigroup in the sense of \cite[Definition~5.1]{kunze2011}. Note that the assumption that
the orbits are $\sigma$-continuous at $0$ made there is only needed to ensure that (i) the semigroup is weakly measurable (which we assume here) and (ii) to ensure that weak integration of Radon-measure-valued functions yields a Radon-measure as result (which is not needed here, as on a Polish space, every measure is a Radon measure).

That the semigroup is integrable means that for every complex $\lambda$ with $\Re\lambda>0$ there is an operator $R(\lambda) \in \mathscr{L}(C_b(\Omega))$ that is $\sigma(C_b(\Omega), \cM(\Omega))$-continuous such that
\[
\langle R(\lambda)f, \mu\rangle = \int_0^\infty  e^{- t \lambda} \langle T_t f, \mu\rangle\, dt
\]
for all $f\in C_b(\Omega)$ and $\mu \in \mathscr{M}(\Omega)$. One can show that $(R(\lambda))_{\Re\lambda>0}$ is a pseudoresolvent (see \cite[Prop.\ 5.2]{kunze2011}), so that whenever it is injective, it is the resolvent of an operator
$A$, i.e., $R(\lambda) = R(\lambda, A) = (\lambda- A)^{-1}$. It is here where we benefit from working on $C_b(\Omega)$:
the same construction could be done on $B_b(\Omega)$ but, in general, the pseudoresolvent obtained on $B_b(\Omega)$ does not turn out to be injective. This is already the case for the semigroup generated by the Laplacian (with, say, Neumann boundary conditions) on a domain $\Omega$, where any function that is zero almost everywhere is in the kernel of the pseudoresolvent.

On $C_b(\Omega)$, however, $R(\lambda)$ is often injective, and in this case the operator $A$ such that $R(\lambda) = (\lambda- A)^{-1}$ is called the \emph{generator of $\mathscr{T}_C$}. Thus, the generator is defined as the unique operator whose resolvent is the Laplace transform of the semigroup. However, if the orbits of the semigroup are not only measurable but even continuous in a certain sense, then this definition is equivalent with the `differential definition' of the generator as derivative of the semigroup at zero, see \cite[Thm.\ 2.10]{kunze2009}. It is a consequence of \cite[Prop.\ 5.7]{kunze2011} that $f\in \fix \mathscr{T}$ if and only if $f\in \ker A$. 

We should point out that the resolvent of $A$ consists of kernel operators, whence one may consider the $\sigma(C_b(\Omega), \cM(\Omega))$-adjoint operators $R(\lambda, A)'$ on $\mathscr{M}(\Omega)$. This is the Laplace transform of the $\sigma(C_b(\Omega), \cM(\Omega))$-adjoint semigroup
$\mathscr{S} \coloneqq \mathscr{T}_C'$ and if $A$ is $\sigma(C_b(\Omega), \cM(\Omega))$-densely defined, then $R(\argument, A)'$ is also a resolvent (see \cite[Prop.\ 2.7]{kunze2009}). In fact, one has $R(\lambda, A)' = R(\lambda, A')$, where $A'$ is the $\sigma(C_b(\Omega), \cM(\Omega))$-adjoint of the operator $A$. Thus, $A'$ is the generator of $\mathscr{T}_C'$ and $\fix \mathscr{T}_C' = \ker A'$. Thus, Theorem \ref{t.main-cont-functions} may be reformulated in this situation.
Note that since $(C_b(\Omega), \beta_0)'= \cM(\Omega)$, the Hahn--Banach theorem yields that the $\beta_0$-closure of a convex set is the same as the $\sigma(C_b(\Omega), \cM(\Omega))$-closure of that set.

\begin{cor}\label{c.main}
	Assume that $\mathscr{T}_C= (T_t^C)_{t\in (0,\infty)} \subseteq \mathscr{L}(C_b(\Omega))$ is a bounded, weakly measurable semigroup of positive kernel operators that contains a strongly Feller operator and has a $\sigma(C_b(\Omega), \cM(\Omega))$-densely defined generator $A$. If $\ker A'$ separates $\ker A$, then $C_b(\Omega) = \fix \mathscr{T}_C \oplus \overline{\lh}^{\beta_0}(I-\mathscr{T}_C)$. Denote by $P$ the projection onto $\fix \mathscr{T}$ along $\overline{\lh}^{\beta_0}(I-\mathscr{T})$. Then, in addition:
	\begin{enumerate}[\upshape (a)]
		\item For every $f\in C_b(\Omega)$ we have $T_t^Cf \to Pf$ uniformly on compact subsets of $\Omega$ as $t \to \infty$.
		
		\item For every $\mu \in \mathscr{M}(\Omega)$ we have $(T_t^C)'\mu \to P'\mu$ in total variation norm as $t\to \infty$.
	\end{enumerate}
\end{cor}

\begin{rem}\label{r.convto0}
	Assume that $\ker A =\{0\}$. By Remark \ref{r.converse-separation}, $\ker A$ separates $\ker A'$, so also $\ker A' =\{0\}$.
	By Theorem~\ref{t.main-cont-functions} it thus follows that (a) and (b) of Corollary~\ref{c.main} hold with $P=0$.
\end{rem}

\section{Doob's theorem revisited} \label{s.doobs-theorem}

In this section we return to the situation where $(\Omega, \Sigma)$ is any measurable space with a countably generated $\sigma$-algebra and generalize the convergence theorem of Doob mentioned in the introduction.

\begin{thm} \label{t.doob-reloaded}
	Let $(\Omega, \Sigma)$ be a measurable space whose $\sigma$-algebra  is countably generated. Let $\mathscr{S} = (S_t)_{t \in (0,\infty)}$ be a bounded semigroup of positive kernel operators on $\mathscr{M}(\Omega)$ and denote the kernel associated with $S_t$ by $k_t$. Assume that $\mu$ is a positive, finite, $\mathscr{S}$-invariant measure, i.e., $S_t\mu = \mu$ for all $t>0$, and that for some time $t_0>0$ and every $x\in \Omega$ the measure $k_{t_0}(x, \cdot)$ is absolutely continuous with respect to $\mu$.
	
	Then, for every $\nu \in \mathscr{M}(\Omega)$, $S_t\nu$ converges in total variation norm as $t\to \infty$.  
\end{thm}
\begin{proof}
	Since $\mu$ is a fixed point of $\cS$, the band $\{\mu\}^{\perp\perp}$ is invariant under $\cS$, and according to Lemma~\ref{l.amcompact}(a) the operator $S_{t_0}$ maps $\mathscr{M}(\Omega)$ into the band $\{\mu\}^{\bot\bot}$. So it suffices to show that the restriction of $\mathscr{S}$ to $\{\mu\}^{\bot\bot}$ converges strongly.
	
	The restriction of $S_{t_0}$ to $\{\mu\}^{\bot\bot}$ is AM-compact by Lemma~\ref{l.amcompact}(c), and the vector $\mu$ is a quasi-interior point in the Banach lattice $\{\mu\}^{\bot\bot}$. Hence, it follows from Theorem~\ref{t.gg19} that the restriction of $\mathscr{S}$ to $\{\mu\}^{\bot\bot}$ indeed converges strongly as time tends to infinity.
\end{proof}

Note that the fixed space of the semigroup $\cS$ in Theorem~\ref{t.doob-reloaded} can be of arbitrary dimension. If, however, the measures $k_{t_0}(x,\argument)$ are mutually absolutely continuous, then $\fix \mathscr{S}$ is one-dimensional. This is discussed in more detail in the following corollary. Two positive measures $\nu_1$ and $\nu_2$ are called \emph{equivalent} if $\nu_1$ is absolutely continuous with respect to $\nu_2$ and vice versa.

\begin{cor}[Doob]
	\label{c.doob-reloaded}
	Let $(\Omega, \Sigma)$ be a measurable space whose $\sigma$-algebra is countably generated. Let $\mathscr{S} = (S_t)_{t \in (0,\infty)}$ be a bounded semigroup of positive kernel operators on $\mathscr{M}(\Omega)$ and denote the kernel associated with $S_t$ by $k_t$. Assume that $\mathscr{S}$ possesses an (a priori not necessarily positive), finite, invariant measure $\mu \not= 0$ and that for some time $t_0>0$ and every $x\in \Omega$ all the measures $k_{t_0}(x,\argument)$ are mutually equivalent.
	Then $\fix \mathscr{S}$ is one-dimensional (thus, spanned by $\mu$) and, for every $\nu \in \mathscr{M}(\Omega)$, $S_t\nu$ converges in total variation norm to a (possibly zero) multiple of $\mu$ as $t\to \infty$.
\end{cor}
\begin{proof}
	According to Lemma~\ref{l.ex-of-positive-fixed-point} there exists a positive, non-zero $\cS$-invariant measure $0 \le \hat \mu \in \cM(\Omega)$. If $\hat \mu(A) =0$, then Formula~\eqref{eq:kernel-operator-on-measures} and the identity $S_{t_0}\hat \mu = \hat \mu$ yield $k_{t_0}(x, A)=0$ for $\hat \mu$-almost every $x\in \Omega$. But since all measures $k_{t_0}(x, \cdot)$ are mutually equivalent, it follows that $k_{t_0}(x, A) =0$ for all $x\in \Omega$. Thus all measures $k_{t_0}(x, \cdot)$ are absolutely continuous with respect to $\hat \mu$ and Theorem~\ref{t.doob-reloaded} shows that, for each $\nu \in \mathscr{M}(\Omega)$, $S_t\nu$ converges in total variation norm as $t \to \infty$.
	
	Let $P$ denote the strong limit of $S_t$ as $t \to \infty$. Then $P$ is a positive projection onto $\fix \mathscr{S}$, and it only remains to show that $P$ has rank $1$.
	
	To this end, note that the range $P\mathscr{M}(\Omega)$ of $P$ is itself a Banach lattice with respect to the order inherited from $\mathscr{M}(\Omega)$ (see \cite[Prop.\ III.11.5]{schaefer1974}). Of course, $\hat \mu$ is an element $P\mathscr{M}(\Omega)$; let $\nu$ be another positive non-zero measure in this range. As above, it follows from formula~\eqref{eq:kernel-operator-on-measures} that, the measure $\nu =S_{t_0}\nu$ is equivalent to $k_{t_0}(x,\argument)$ for each $x \in \Omega$. 
	In particular, the measures $\hat \mu$ and $\nu$ are equivalent. Consequently, they generate the same closed ideal in $\mathscr{M}(\Omega)$, and hence they also generate the same closed ideal in the Banach lattice $P\mathscr{M}(\Omega)$. This shows that every non-zero positive element of $P\mathscr{M}(\Omega)$ is actually a quasi-interior point within this space. Hence, $P\mathscr{M}(\Omega)$ is one-dimensional.
\end{proof}

The history of Doob's theorem has already beeen discussed in the introduction. Here, it is worth adding that for a long time merely the classical setting in which $\Omega$ is a Polish space, endowed with its Borel $\sigma$-algebra, was studied. Only recently -- in \cite{kulik_scheutzow2015} -- the case of an arbitrary measure space with countably generated $\sigma$-algebra was addressed. It is also worthwhile to note that the authors of \cite{kulik_scheutzow2015} actually established a version of Doob's theorem in the time-discrete setting and then obtained the time-continuous result as a corollary (see \cite[Sect.\ 2]{gerlach_glueck2018} for a general account on how to pass from discrete to continuous time). For the more general result in Theorem \ref{t.doob-reloaded}, which appears to be new, a time-discrete analogue cannot hold. This can already be seen in the following simple example: 

Let $\Omega \coloneqq \{1,2\}$ be endowed with the discrete $\sigma$-algebra. For the Markov operator
\[ S\coloneqq
\begin{pmatrix}
	0 & 1 \\
	1 & 0
\end{pmatrix}
\]
on $\mathscr{M}(\Omega)$ all the assumptions of Theorem~\ref{t.doob-reloaded} are satisfied for the measure
$\mu(\{1\}) \coloneqq \mu(\{2\}) \coloneqq \frac{1}{2}$ and the discrete semigroup $(S^n)_{n\in\N}$. However, this semigroup is not strongly convergent.

\section{Applications} \label{s.applications}

Now the abstract results from the previous sections will be applied to various semigroups associated with (systems of) partial differential equations and information about their asymptotic behavior will be obtained. Subsection~\ref{s.appl-whole-space} starts with paraboilic equations on $\R^d$ with possibly unbounded coefficients. This can, by now, be considered a `classical topic', initiated in \cite{metafune_pallara_wacker2002}, where also the asymptotic behavior of the associated semigroup was addressed. However, 
in \cite{metafune_pallara_wacker2002}, the authors limited themselves to the situation of the classical Theorem of Doob whereas we can characterize the asymptotic behavior in full detail (see
Theorem~\ref{t.appl-whole-space}). 

If, instead of all of $\R^d$, we consider the same equations on an unbounded domain and subject to Dirichlet boundary conditions, then much less is known about the asymptotic behavior. From a heuristic point of view, one would expect the semigroup to converge -- in one way or another -- towards zero. However, to the best of our knowledge, the results in Subsection \ref{s.appl-unbdd-domains} are the first in this direction.

The remaining Subsections~\ref{s.appl-irred-systems} and~\ref{s.appl-red-systems} are devoted to systems of parabolic equations on $\R^d$ that are coupled by matrix-valued potentials. As mentioned in the introduction, these equations illustrate the fact that our results can be applied to bounded, rather than contractive, semigroups. Systems of this form were considered in \cite{addona_angiuli_lorenzi2019}. To tackle the asymptotic behavior of such systems though, the authors of \cite{addona_angiuli_lorenzi2019} made additional regularity and growth assumptions. An application of our abstract results shows that the assumptions imposed to show well-posedness of the equation, i.e.\ existence of the semigroup, are already sufficient to treat the asymptotic behavior -- no additional assumptions are needed (Subsection~\ref{s.appl-irred-systems}). Moreover, our results can be generalized to the non-irreducible setting without additional effort (Subsection \ref{s.appl-red-systems}).

\subsection{Parabolic equations on the whole space \texorpdfstring{$\R^d$}{}} \label{s.appl-whole-space}

In this section, we revisit second order differential operators in non-divergence form on all of $\R^d$ which were studied first in the seminal paper \cite{metafune_pallara_wacker2002}. These operators are formally given by

\begin{equation}
	\label{eq.operator}
	Au = \sum_{i,j=1}^d a_{ij}D_{ij}u + \sum_{j=1}^d b_jD_ju.
\end{equation}
As  in \cite{metafune_pallara_wacker2002}, the following assumptions on the coefficients will be imposed in what follows.

\begin{hyp}\label{hyp1}
	For some $\alpha>0$ the real-valued coefficients $a_{ij}, b_j$ are locally $\alpha$-H\"older continuous for all $i,j=1, \ldots, d$. The diffusion coefficients are assumed to be symmetric ($a_{ij}=a_{ji}$) and to satisfy the ellipticity condition
	\[
		\sum_{i,j=1}^d a_{ij}(x) \xi_i\xi_j \geq \eta(x) |\xi|^2
	\]
	for all $x\in \R^d$ and $\xi\in \R^d$, where $\eta : \R^d\to (0,\infty)$ is such that $\inf_K\eta >0$ for every compact set $K\subseteq \R^d$.
\end{hyp}

It may happen that $\inf_{\R^d}\eta =0$ or that the coefficients are unbounded. In order to define a realization of the differential operator \eqref{eq.operator} on the space $C_b(\R^d)$, a suitable domain is needed. Set
\[
	D_{\mathrm{max}} \coloneqq \Big\{ u\in C_b(\R^d)\cap\bigcap_{1<p<\infty} W^{2,p}_{\mathrm{loc}}(\R^d) : Au \in C_b(\R^d)\Big\}.
\]
As it turns out, it may happen that the elliptic equation $\lambda u - A u = f$ is not well-posed on $D_{\mathrm{max}}$, as the solution may not be unique; see \cite[Sect.\ 7]{metafune_pallara_wacker2002} for examples. Nevertheless, it is proved 
in \cite[Sect.\ 3]{metafune_pallara_wacker2002} that, whenever $f\geq 0$, there is a \emph{minimal} solution of this equation.

\begin{prop}\label{p.elliptic}
	There exists a subspace $\hat D$ of $D_\mathrm{max}$ such that the realization of $A$ on $\hat D$, which is denoted by 
$\hat A$ in what follows, has the following properties:
	\begin{enumerate}[(a)]
		\item $\hat A$ is a closed operator with $(0,\infty) \subseteq \rho (\hat A)$.
		\item For every $\lambda >0$ the resolvent $R(\lambda, \hat A)$ is a positive contraction on $C_b(\R^d)$.
		\item For every $\lambda >0$ and $f\geq 0$ the function $R(\lambda, \hat A)f$ is the minimal solution of the equation $\lambda u - Au = f$ in $D_\mathrm{max}$.
	\end{enumerate}
\end{prop}

\begin{proof}
	See \cite[Thm.\ 3.4 and Prop.\ 3.6]{metafune_pallara_wacker2002}, and also \cite[Sect.\ I.1.1]{lorenzi2017}.
\end{proof}

As it turns out, the operator $\hat A$ is the generator of a semigroup $\mathscr{T}$:

\begin{prop}\label{p.parabolic}
	The operator $\hat A$ is the generator of a positive, strongly Feller contraction semigroup $\mathscr{T} = (T_t)_{t \in (0,\infty)}$ on the space $C_b(\R^d)$. Moreover, there is a measurable function $p: (0,\infty)\times \R^d\times \R^d \to (0,\infty)$ such that
	\begin{equation}\label{eq.representation}
		(T_tf)(x) = \int_{\R^d} p(t,x,y) f(y)\, dy
	\end{equation}
	for every $x \in \R^d$ and $f\in C_b(\R^d)$.
\end{prop}

\begin{proof}
	See \cite[Sect.\ 4]{metafune_pallara_wacker2002} or \cite[Sect.\ I.1.2]{lorenzi2017}.
\end{proof}

The abstract main results of this article now yield the following description of the asymptotic behavior of the semigroup $T$ in terms of $\ker A$ and $\ker A'$.

\begin{thm} \label{t.appl-whole-space}
	For the semigroup $\cT = (T_t)_{t \in (0,\infty)}$ from Proposition~\ref{p.parabolic} exactly one of the following alternatives is true:
	\begin{enumerate}[\upshape (1)]
		\item
		One has $\ker A' \neq \{0\}$. \\
		In this case, there exists a non-zero invariant measure $0 \leq \mu \in \cM(\Omega)$ which is strictly positive in the sense that $\langle \mu, f \rangle > 0$ for each non-zero $0 \le f \in C_b(\R^d)$. One has $\hat D = D_{\max}$, $\ker A = \lh\{\one\}$, $\ker A' =\lh\{\mu\}$. Moreover,
		and
		\begin{align*}
								&  T_tf \to \langle f,\mu\rangle\cdot \one 	\quad \text{locally uniformly} \\
			\text{and} \quad 	&  T_t'\nu \to \langle \one, \nu\rangle\mu 	\quad \text{in total variation}
		\end{align*}
		as $t\to\infty$ for all $f\in C_b(\R^d)$ and $\nu \in \cM (\R^d)$.
		
		\item 
		One has $\ker A =\{0\}$. \\
		In this case,
		\begin{align*}
								&  T_tf \to 0 		\quad \text{locally uniformly} \\
			\text{and} \quad 	&  T_t'\nu \to 0 	\quad \text{ in total variation}
		\end{align*}
		for all $f\in C_b(\R^d)$ and $\nu \in \cM (\R^d)$ as $t\to\infty$.
		
		\item 
		The space $\ker A'$ does not separate $\ker A$. \\
		In this case, there exists $f\in C_b(\R^d)$ such that the orbit $(T_tf)$ does not converge pointwise,
		and there exists $\nu\in \cM(\R^d)$ such that the orbit $(T_t'\nu)$ does not converge in total variation norm.
	\end{enumerate}	
\end{thm}

\begin{proof}
	If~(2) and~(3) fail, then $\ker A'$ is clearly non-zero, i.e.\ (1) holds; this shows that at least one of the three assertions is satisfied. Next we show that the claims within each of the points~(1)--(3) are true.
	
	(1) Assume that (1) holds. Then there exists, due to Lemma~\ref{l.ex-of-positive-fixed-point}, a positive non-zero measure $\mu \in \ker A'$. Strict positivity of $\mu$ follows from strict positivity of the kernel function $p$ and from $\mu = T_1 \mu$. Moreover, Corollary \ref{c.doob-reloaded} (which is applicable due to the representation of the semigroup operators given in Proposition~\ref{p.parabolic}) yields $\dim \ker A' = 1$ and convergence of the orbits of $T'$. Convergence of the orbits of $T$ as well as $\dim \ker A = 1$ thus follow from Theorem \ref{t.main-bdd-functions}. 
	
	It only remains to show that $\one$ is a fixed vector of $\cT$; by \cite[Prop.\ 5.7]{metafune_pallara_wacker2002} this implies $\hat D = D_{\mathrm{max}}$. 
	Let $t \in (0,\infty)$. Then $T_t \one \le \one$ since $\cT$ is a contraction semigroup. 
	As $\langle \mu, \one - T_t \one \rangle = 0$, this yields $T_t \one = \one$ due to the strict positivity of $\mu$.
	
	(2) 
	The claimed convergence in case that $\ker A = \{0\}$ follows from Remark~\ref{r.convto0}. 
	
	(3) The claims in this case follow immediately from Theorem~\ref{t.main-cont-functions}.
	
	The properties that have just been shown to follow from~(1)--(3), imply that the three conditions are mutually exclusive.
\end{proof}

In the literature,  alternative (1) in Theorem \ref{t.appl-whole-space} is best studied and, indeed, this situation is already considered in the reference \cite{metafune_pallara_wacker2002}. Assuming that $\hat{D} = D_{\max}$ the authors of \cite{metafune_pallara_wacker2002} refer to a version of Doob's theorem in \cite[Thm.\ 4.2.1]{daprato_zabczyk1996} that yields pointwise convergence of each orbit of $\cT$ in $C_b(\R^d)$ whenever there exists an invariant probability measure.

The natural question is how to decide whether there exists an invariant probability measure. The result in \cite[Thm.\ 6.3]{metafune_pallara_wacker2002} provides a sufficient condition for this in terms of Lyapunov functions. This is certainly the most common tool to establish the existence of an invariant measure and \cite{metafune_pallara_wacker2002} contains many examples where this result can be applied.\smallskip

We thus turn our attention to the possibilities (2) and (3) of Theorem \ref{t.appl-whole-space} which, to the best of our knowledge, have not been studied systematically in the literature. As for case (2), making use of the results of \cite{metafune_pallara_wacker2002}, we can give a large class of examples where we are in this situation. Indeed, \cite[Thm.\ 3.12]{metafune_pallara_wacker2002} gives a sufficient condition (again in terms of a Lyapunov function) for $\hat D = D_{\mathrm{max}}\cap C_0(\R^d)$ ($=D(A)$ in the notation of \cite{metafune_pallara_wacker2002}). Thus, in this situation $\ker A \subset C_0(\R^d)$ and the maximum principle yield $\ker A=\{0\}$.

\begin{example}
	Consider the operator $Au = \Delta u + \sum_{j=1}^d b_j D_j u$, where the $b_j$ are the functions from Hypothesis~\ref{hyp1}. If 
	\[
		\lim_{|x|\to \infty} \frac{\sum_{j=1}^d b_j (x) x_j}{|x|^2(\log |x|)^{\delta+1}} = \infty
	\]
	for some $\delta >0$, then $\ker A = \{0\}$ and assertion (2) in Theorem \ref{t.appl-whole-space} holds. Indeed, \cite[Cor.\ 3.12]{metafune_pallara_wacker2002} yields $\hat D \subset C_0(\R^d)$ and by the above, $\ker A =\{0\}$ follows.
\end{example}

Assertion (3) of Theorem \ref{t.appl-whole-space} holds in the classical situation of the Laplacian $A=\Delta$. Indeed, in this case $\ker A =\lh\{\one\}$ and $\ker A' = \{0\}$, as any invariant measure must be translation invariant and thus a multiple of the Lebesgue measure (which is infinite). It thus follows from Theorem~\ref{t.appl-whole-space} that there exists a function $f \in C_b(\R^d)$ such that $T_tf$ is not pointwise convergent as $t \to \infty$.
	
	Of course, in this special case much more is known; for instance, one can characterise those $f \in C_b(\R^d)$ for which the orbit $T_tf$ converges pointwise as $t \to \infty$. More information on this, and also on the long-term behavior of general parabolic equations with bounded coefficients, can be found in the survey article \cite{denisov2005}.
\smallskip

While $\dim \ker A' \leq 1$ in the situation of Theorem \ref{t.appl-whole-space}, $\ker A$ can  in fact be multi-dimensional;
 this is illustrated by the following simple example.

\begin{example}
	Let $d = 1$ and $b(x) = \frac{2x}{1+x^2}$ for each $x \in \R$. Consider the operator $A$ given by
	\begin{align*}
		Au = u'' + bu' 
	\end{align*}
As $b$ is bounded, $A$ is a bounded perturbation of the Laplacian whence $\hat D = D_{\max}$.	
The operator $A$ vanishes both on the function $\one$ and on the function $\arctan$, and both of them are contained in $D_{\max} = \hat D$. Hence, $\ker A$ is at least two-dimensional.
	
	Thus, assertion (3) of Theorem~\ref{t.appl-whole-space} holds in this example.
\end{example}

\subsection{Parabolic equations on unbounded domains with Dirichlet boundary conditions} 
\label{s.appl-unbdd-domains}

In this section, we consider second order, strictly elliptic operators with (possibly) unbounded coefficients on an unbounded domain $\Omega\subseteq \R^d$. Naturally, some sort of boundary conditions have to be imposed on $\partial\Omega$. If Neumann boundary conditions are used, then the constant $1$-function satisfies these boundary conditions and the situation is similar to that of the last section. So, instead, we will focus on Dirichlet boundary conditions. If the domain is bounded, then classical theory suggests that the semigroup generated by a second order elliptic operator with bounded coefficients and subject to Dirichlet boundary conditions should converge to 0. If $\Omega$ is unbounded, the situation is not so easy; this is the focus of this section.

Consider elliptic operators of form
\[
	Au = \sum_{i,j=1}^d a_{ij}D_{ij}u + \sum_{j=1}^d b_jD_j u,
\]
and subject to Dirichlet boundary conditions. Such operators were first studied in \cite{fornaro_metafune_priola2004}; see also \cite[Sect. I.12]{lorenzi2017}. 
We work on the Polish space $\overline{\Omega}$ -- and thus on the spaces $C_b(\overline{\Omega})$ and $\cM(\overline{\Omega})$ -- amd make the following assumptions similar to \cite{fornaro_metafune_priola2004}.

\begin{hyp}
	\label{hyp.domain}
	Let $\Omega\subseteq \R^d$ be an open and connected domain with $C^{2+\alpha}$-boundary where $0<\alpha<1$. Moreover, consider coefficients $a_{ij}, b_j : \overline{\Omega} \to \R$ that are locally $\alpha$-H\"older continuous. Assume 
	that the diffusion coefficients $a_{ij}$ are symmetric and uniformly elliptic in the sense that there is a constant $\eta>0$ such that
	\[
		\sum_{i,j=1}^d a_{ij}(x)\xi_i\xi_j\geq \eta|\xi|^2
	\]
	for all $x\in \overline{\Omega}$ and $\xi\in \R^d$. Finally, assume that there exists a function $V \in C^2(\overline{\Omega})$
	such that $\lim_{|x|\to \infty} V(x) =\infty$ and such that $(\lambda_0-A)V\geq 0$ for some $\lambda_0 \ge 0$.
\end{hyp}

Note that that in \cite{fornaro_metafune_priola2004}, the coefficients were additionally assumed to be differentiable. However, that assumption is only needed to establish gradient estimates for the semigroup which are not needed for our purposes. The existence of the function $V$ in Hypothesis \ref{hyp.domain} ensures that, in the notation of the last section, $D_{\max} = \hat D$. We thus endow the operator $A$ with the domain
\[
	D(A) \coloneqq \Big\{ u \in C_b(\overline{\Omega}) \cap \bigcap_{1<p<\infty} W^{2,p}_{\mathrm{loc}}(\Omega) :
	Au \in C_b(\overline{\Omega}), u(x) = 0 \mbox{ for } x\in \partial\Omega\Big\}.
\]
Here, $Au\in C_b(\overline{\Omega})$ means that there is a function $f\in C_b(\overline{\Omega})$ such that
$Au=f$ on $\Omega$.

\begin{prop}
	\label{p.domaingen}
	Under the above assumptions, the operator $A$ is the generator of positive, strongly Feller
	contraction semigroup $\cT = (T_t)_{t\in (0,\infty)}$
	on the space $C_b(\overline{\Omega})$. 
	There is a measurable function $p: (0,\infty)\times \overline{\Omega}\times \overline{\Omega} \to [0,\infty)$ such that
	\begin{equation}
		\label{eq.rep}
		(T_tf)(x) = \int_{\overline{\Omega}} p(t,x,y) f(y) \dx y,
	\end{equation}
	for every $x \in \overline{\Omega}$, $t>0$ and $f \in C_b(\overline{\Omega})$.
	Moreover, for all $t > 0$ one has $p(t,x,y) > 0$ for $x,y \in \Omega$ and $p(t,x,y) = 0$ for $x \in \partial \Omega$ and $y \in \overline{\Omega}$.
\end{prop}

\begin{proof}
	See \cite[Sect.\ 2]{fornaro_metafune_priola2004} or \cite[Sect.\ I.12.1]{lorenzi2017}.	
\end{proof}

The following gives a characterization of the asymptotic behavior of the semigroup $T$ similar to Theorem \ref{t.appl-whole-space}. In a way, the Dirichlet boundary conditions are responsible for the fact that assertion (1) of that theorem cannot occur.

\begin{thm}\label{t.appl-dirichlet}
	For the semigroup $\cT=(T(t))_{t>0}$ from Proposition \ref{p.domaingen} we have $\ker A'= \{0\}$. Moreover:
	\begin{enumerate}[\upshape(a)]
		\item 
		If $\ker A = \{0\}$,
		then for each $f \in C_b(\overline{\Omega})$ one has $T_tf \to 0$ locally uniformly on $\overline{\Omega}$ as $t \to \infty$ and for each $\nu \in \cM (\overline{\Omega})$ one has $T_t'\nu \to 0$ in total variation as $t\to\infty$.

		\item 
		If $\ker A \neq \{0\}$, 
		then there is $f\in C_b(\overline{\Omega})$ such that the pointwise limit of the orbit $(T_tf)$ does not exist and 
		there is $\mu \in \cM (\overline{\Omega})$ such that the orbit $(T_t'\mu)$ does not converge in total variation norm.
	\end{enumerate}
\end{thm}

\begin{proof}
	Assume towards a contradiction that $\ker A' \not= \{0\}$. 
	Then by Lemma~\ref{l.ex-of-positive-fixed-point} there exists a non-zero fixed point $\mu \ge 0$ of the dual semigroup $\cT'$.
	It follows from the boundary behaviour of the kernel function $p$ (Proposition~\ref{p.domaingen}) 
	that $\mu(\Omega) > 0$ and hence, by the strict positivity of $p$ in $\Omega$, $\langle \mu, f \rangle > 0$
	for each non-zero function $0 \le f \in C_b(\overline{\Omega})$.
	
	For each $t \in (0,\infty)$ one has $T_t \one \le \one$ by the contractivity of $\cT$; 
	since $\langle \mu, \one - T_t \one \rangle = 0$, the aforementioned property of $\mu$ implies
	that $T_t \one = \one$, and thus $\one \in \ker A \subset D(A)$, which is a contradiction. 
	
	The claimed properties in~(a) and~(b) now follow from Theorem~\ref{t.main-cont-functions} and Proposition~\ref{p.weak-convergence-of-sg}.
\end{proof}

According to the previous theorem, a criterion to ensure $\ker A =\{0\}$ is needed to decide about the asymptotic behavior of the semigroup.

\begin{prop}
	\label{p.liouville-dirichlet}
	Assume that the number $\lambda_0$ in Hypothesis~\ref{hyp.domain} can be chosen as $\lambda_0 = 0$. Then $\ker A = \{0\}$.
\end{prop}

\begin{proof}
	Define $K := \inf_{x \in \overline{\Omega}} V(x)$; then 
	$K > -\infty$ since $V(x) \to \infty$ as $\lvert x\rvert \to \infty$.
	
	Fix $u \in \ker A$. It suffices to show that $u \le 0$ on $\Omega$. Let $\varepsilon > 0$ and consider the function $u - \varepsilon V: \overline{\Omega} \to \R$. This function tends to $-\infty$ as $\lvert x\rvert \to \infty$, so it attains its maximum $m$ at a point $x_\varepsilon \in \overline{\Omega}$.
	
	Moreover, the function $u - \varepsilon V$ is \emph{$A$-subharmonic}, i.e., $A(u-\varepsilon V) \ge 0$. We argue that $x_\varepsilon \in \partial \Omega$. Indeed, choose a sufficiently large radius $R$ such that $u-\varepsilon V$ is smaller than $m-1$ outside $\Omega \cap B(0,R)$. Then $\lvert x_\varepsilon\rvert < R$ and hence, $x_\varepsilon \in \overline{\Omega \cap B(0,R)}$, so the restriction of $u-\varepsilon V$ to $\overline{\Omega \cap B(0,R)}$ attains its maximum at $x_\varepsilon$. It thus follows from the maximum principle \cite[Thm.\ 2]{bony1967} that $x_\varepsilon$ is located at the boundary of $\Omega \cap B(0,R)$. Since $\lvert x_\varepsilon \rvert < R$, it follows that actually $x_\varepsilon \in \partial \Omega$.
	
	Hence $u(x_\varepsilon) = 0$, so for every $x \in \Omega$ one has
	\begin{align*}
		(u - \varepsilon V)(x) \le (u-\varepsilon V)(x_\varepsilon) = -\varepsilon V(x_\varepsilon) \le -\varepsilon K, 
	\end{align*}
	and thus $u(x) \le \varepsilon (V(x) - K)$. Since $\varepsilon$ was arbitrary, it follows that $u(x) \le 0$ for each $x \in \Omega$.
\end{proof}

As a first illustration, it is illuminating to apply Proposition~\ref{p.liouville-dirichlet} to the Laplace operator.

\begin{example}
	\label{ex:laplacian-on-ex-dom}
	Consider $A = \Delta$ on an unbounded set $\Omega \subsetneq \R^d$ whose boundary satisfies the smoothness assumption from Hypothesis \ref{hyp.domain}. Note that $V(x) = \lvert x\rvert^2 + 2d$ and $\lambda_0 = 1$ satisfy the assumptions of Hypothesis \ref{hyp.domain}. As it turns out, the long-time behavior of the semigroup depends on the dimension $d$:
	
	\begin{enumerate}[(a)]
		\item Let $d \in \{1,2\}$. Then, for each $f\in C_b(\overline{\Omega})$, $T_tf$ converges to $0$, uniformly 
		on compact subsets of $\overline{\Omega}$, as $t \to \infty$.
		
		\item If $d \ge 3$ and $\Omega = \R^3 \setminus \overline{B(0,r)}$ for some $r>0$, then the semigroup $\cT$ on $C_b(\overline{\Omega})$ is not convergent to $0$ as $t \to \infty$.
	\end{enumerate}
\end{example}
\begin{proof}
	(a): We may, and shall, assume that $0$ is in the interior of $\R^d \setminus \Omega$. By Theorem \ref{t.appl-dirichlet} and Proposition \ref{p.liouville-dirichlet}, it suffices to construct a function $\tilde V$ as in 
	Hypothesis~\ref{hyp.domain} where $\lambda_0$ can be chosen as $0$.
	
	Let $\tilde V: \R^d \setminus \{0\} \to \R$ denote the Newton potential. Since $0$ is not in the boundary of $\Omega$, it follows that $\tilde V \in C^2(\overline{\Omega})$. Clearly, $-A\tilde V = -\Delta \tilde V = 0$. Moreover, $\tilde V(x) \to \infty$ as $\lvert x \rvert \to \infty$ since $d \le 2$. 
	
	(b) This time, the Newton potential does not serve as a Lyapunov function, but as a counterexample; denote it again by $\tilde V$. Then, for an appropriately chosen number $c > 0$, the function $\tilde V + c \one$ vanishes on the boundary of $B(0,r)$, and it is bounded since $d \ge 3$. Hence, $\tilde V + c \one \in D(A)$ and thus $\tilde V + c \one \in \ker A$, so this function is a fixed vector of the semigroup $\cT$. Hence, one does not have convergence to $0$.
\end{proof}

Concerning Example~\ref{ex:laplacian-on-ex-dom}(b) the following two remarks are in order.

\begin{rems}
	(a) Example~\ref{ex:laplacian-on-ex-dom}(b) can be extended to more general exterior domains $\Omega =\R^d\setminus K$, where $K$ is a compact set. Indeed, pick $r>0$ large enough, so that $K \subset B(0,r)$. With a bit of effort, it can be shown that the heat semigroup on $\R^d\setminus\overline{B(0,r)}$ is dominated by the heat semigroup on $R^d\setminus K$. If the latter would converge to $0$, then so would, by domination, the former---a contradiction to 
	Example~\ref{ex:laplacian-on-ex-dom}(b).

	(b) There is a probabilistic interpretation of the distinction between the cases $d \le 2$ and $d \ge 3$ in Example~\ref{ex:laplacian-on-ex-dom}: in dimensions $1$ and $2$ Brownian motion is recurrent, i.e., the probability to eventually hit the complement of $\Omega$ is always $1$; since each particle is killed on the boundary due to the Dirichlet boundary conditions, the process loses mass, which explains the convergence to $0$. In dimension $\ge 3$ on the other hand, Brownian motion is transient, so the probability that a trajectory never hits the complement of $\Omega$ is non-zero.
\end{rems}

Another simple but instructive example is the Laplace operator with constant drift on the half line:

\begin{example}
	Let $\Omega = (0,\infty) \subseteq \R$ and let $A$ be given by $Au = u'' + bu'$ for a fixed number $b \in \R$. One can choose $V(x) = x+b$ and $\lambda_0 = 1$ in Hypothesis~\ref{hyp.domain}, so the theory of this section applies. Now distinguish two situations:
	
	(a) If $b > 0$, then the bounded function $u$ given by $u(x) = 1 - e^{-bx}$ for $x \in \Omega$ is in the kernel of $A$, so we do not have convergence of the semigroup.
	
	(b) If $b \le 0$, then one can use the Lyapunov function $\tilde V(x) = x$ rather than $V$ and $\lambda_0 = 0$ rather than $\lambda_0 = 1$ in Hypothesis~\ref{hyp.domain}. Hence, Proposition~\ref{p.liouville-dirichlet} implies that $\ker A = \{0\}$, so the orbits of semigroup on $C_b(\overline{\Omega})$ that is associated with $A$ converge to $0$ uniformly on compact subsets of $\overline{\Omega}$.
\end{example}

The last example in this subsection is the Ornstein--Uhlenbeck semigroup on domains with Dirichlet boundary conditions.

\begin{example}
	Assume that the domain $\Omega \subseteq \R^d$ has sufficiently smooth boundary as described in Hypothesis~\ref{hyp.domain} and that $0$ is not contained in its closure. Consider the Ornstein--Uhlenbeck operator $A$ given by $Au(x) = \Delta u(x) - \langle x, \nabla u(x) \rangle$. It is easy to see that  Hypothesis~\ref{hyp.domain} is satisfied; indeed simply choose the Lyapunov function $\overline{\Omega} \ni x \mapsto \lvert x \rvert^2 + 1 \in \R$ and a sufficiently large number $\lambda_0 > 0$. Hence, Proposition~\ref{p.domaingen} shows that $A$ generates a semigroup $\cT = (T_t)_{t \in (0,\infty)}$ on $C_b(\overline{\Omega})$.
	
	To show that the orbits of the semigroup converge to $0$, we construct a new Lyapunov function $V$ such that $\lambda_0$ can be chosen as $0$ (Proposition~\ref{p.liouville-dirichlet}). Let $\tilde V: \R^d \setminus \{0\} \to \R$ denote the Newton potential.  First consider the case $d\geq 3$. Then $\Delta \tilde V=0$ and  $\langle x, \nabla \tilde V(x)\rangle = c_d \vert x \vert^{2-d}$ for a $d$-dependent constant $c_d$.
	
	Now define the Lyapunov function $V$ as $V(x) = \lvert x\rvert^2 + r \tilde V(x)$ for a large number $r > 0$. Then
	\begin{align*}
		AV(x) = 2d - 2 \lvert x \rvert^2 - r c_d \lvert x \vert^{2-d} \quad \text{for} \quad x \in \overline{\Omega};
	\end{align*}
	So if $r$ has been chosen sufficiently large, $AV \le 0$ in $\overline{\Omega}$. Hence, Proposition~\ref{p.liouville-dirichlet} implies that $\ker A = \{0\}$, and Corollary~\ref{c.main} shows that, for each $f \in C_b(\overline{\Omega})$, $T_tf \to 0$ uniformly on compact subsets of $\overline{\Omega}$ as $t \to \infty$. 
	
	For the cases $d=1, 2$, one easily verifies that $A\tilde V \leq 0$ on $\Omega$, so that in these cases $\tilde V$ itself can be used as a Lyapunov function.
\end{example}

\subsection{Irreducible systems coupled by matrix potentials} \label{s.appl-irred-systems}

This section is about the asymptotic behavior of coupled systems of parabolic equations recently considered in 
\cite{addona_angiuli_lorenzi2019}. Formally, these equations are governed by $m\geq 2$ elliptic operators on $\R^d$ that are coupled only via the terms of order zero. To be more precisely, consider the operator $\bA$ (we follow the convention of \cite{addona_angiuli_lorenzi2019} and denote vector-valued objects with boldface letters), given by
\begin{align*}
	\bA\bu & \coloneqq \sum_{i,j=1}^d a_{ij}(x)D_{ij}\bu (x) + \sum_{j=1}^d b_j(x)D_j\bu (x) +\bC(x)\bu(x),
\end{align*}
where $a_{ij}, b_j : \R^d \to \R$ are scalar functions, where $\bC : \R^d \to \R^{m\times m}$ is matrix-valued and where $\bu : \R^d \to \R^m$. Note that the coupling of the $m$ components is only via the matrix-valued function $\bC$. Skipping this potential term, one obtains a diagonal operator and, indeed, the scalar operator 
\begin{equation}\label{eq.scalara}
	A \coloneqq \sum_{i.j=1}^d a_{ij}D_{ij} + \sum_{j=1}^d b_jD_j
\end{equation}
plays an important role in \cite{addona_angiuli_lorenzi2019}.

The following standing assumptions are imposed on the coefficients.

\begin{hyp}\label{hyp2}
	Assume that, for some $\alpha>0$, the coefficients $a_{ij}, b_j$ ($i,j=1, \ldots, \infty$) and $c_{kl}$ ($k,l=1, \ldots, m$) are locally $\alpha$-H\"older continuous. The coefficients $a_{ij}$ are assumed to be symmetric (i.e., $a_{ij}=a_{ji}$ for all $i,j$) and strictly elliptic in the sense that for some $\eta>0$ one has
	\[
		\sum_{i,j=1}^d a_{ij}(x)\xi_i\xi_j \geq \eta |\xi|^2
	\]
	for all $x\in \R^d$ and $\xi\in \R^d$. Moreover, assume the following:
	\begin{enumerate}[(a)]
		\item There is a function $\varphi \in C^2(\R^d)$ satisfying $\lim_{|x|\to\infty} \varphi (x) = \infty$, such that for some constants $\alpha, \beta>0$, one has $A\varphi \leq \alpha - \beta \varphi$.
		\item The matrix $\bC = (c_{kl})$ satisfies:
		\begin{enumerate}[(i)]
			\item For all $x\in \R^d$ and $y\in \R^m$ the inequality $\langle \bC(x)y,y\rangle \leq 0$ holds.
			\item $c_{kl} \geq 0$ for $k\neq l$.
			\item There is a vector $\xi \in \R^m\setminus \{0\}$ such that $\xi \in \bigcap_{x\in \R^d} \ker \bC(x)$.
			\item If $K \subseteq \{1, \ldots, m\}$ is such that $c_{kl}\equiv 0$ whenever $k\in K$ and $l\not\in K$, then $K=\emptyset$ or $K=\{1, \ldots, m\}$.
		\end{enumerate}
	\end{enumerate}
\end{hyp}

Under the above assumptions, the intersection $\bigcap_{x\in \R^d}\ker \bC(x)$ is one-dimensional and we may (and shall) assume that $\xi \ge 0$ throughout the rest of the section; these observations are established in the proof of \cite[Prop.\ 3.2]{addona_angiuli_lorenzi2019}.

\begin{rem}
	Note that under the assumptions of Hypothesis \ref{hyp2} the scalar operator $A$, defined by \eqref{eq.scalara} satisfies
Hyposthesis \ref{hyp1}. Moreover, by assumption~(a) of Hypothesis~\ref{hyp2} $\varphi$ acts as a Lyapunov function for that operator $A$. By
 \cite[Thm.\ 6.3]{metafune_pallara_wacker2002} this implies that (adopting the notation from  Section \ref{s.appl-whole-space}) $\hat D = D_\mathrm{max}$ and that there exists an invariant measure for the scalar semigroup $T = (T_t)_{t\geq 0}$ generated by $A$.
\end{rem}

Under the assumptions in Hypothesis~\ref{hyp2}, it is proved in \cite{addona_angiuli_lorenzi2019} that there is a semigroup
$\bT = (\bT_t)_{t\geq 0}$ on the space $C_b(\R^d; \R^m)$, such that for every $\bff \in C_b(\R^d; \R^m)$ the function
$\bu (t, x) \coloneqq (\bT_t \bff)(x)$ is the unique bounded, classical solution of the Cauchy problem
\[
	\begin{cases} 
		D_t\bu = \bA\bu & \mbox{ on } (0,\infty)\times \R^d, \\
		\bu(0, x) = \bff (x)& \mbox{ for } x\in \R^d.
	\end{cases}
\]
In order to apply our abstract results, we rewrite this system as follows. Set
\[
	\Omega \coloneqq \{1, \ldots, m\}\times \R^d \subseteq \R^{d+1}
\]
and identify the vector-valued function $\bff = (f_1, \ldots, f_m): \R^d \to \R^m$ with the scalar-valued function $f: \Omega \ni (k,x) \mapsto f_k(x) \in \R$. Likewise, define the semigroup $\cT = (T_t)_{t\geq 0}$ on $C_b(\Omega)$ by setting
\[
	(T_tf)(k,x) \coloneqq (\bT_t \bff )_k(x).
\]
This semigroup has the following properties:

\begin{prop}\label{p.sg}
	Under the assumptions of Hypothesis~\ref{hyp2}, $\cT$ is a bounded semigroup of positive operators on the space $C_b(\Omega)$. Moreover, the semigroup consists of strongly Feller operators.
\end{prop}

\begin{proof}
	The authors of \cite{addona_angiuli_lorenzi2019} and their coauthors established the existence of the semigroup in earlier works, see \cite{delmonte_lorenzi2011, addonna_et_al2017}. In particular, it follows from \cite[Thm.\ 3.2]{addonna_et_al2017} that the semigroup consists of strongly Feller operators. The positivity of the semigroup is established in \cite[Prop.\ 2.8]{addona_angiuli_lorenzi2019} making use of Hypothesis \ref{hyp2}(b)(ii). It is also proved in \cite{addona_angiuli_lorenzi2019} (see Equation (2.8) in that article) that the semigroup $\bT$ is contractive. However, the authors used on $\R^m$ the Euclidian norm $|y| = \sqrt{y_1^2+\ldots + y_m^2}$ and on $C_b(\R^d; \R^m)$ the induced supremum norm. In contrast to that we use, on the scalar space $C_b(\Omega)$, the usual scalar supremum norm -- which corresponds to using the norm $|y|_\infty \coloneqq \max \{|y_1|, \ldots, |y_m|\}$ on $\R^m$. Therefore, with respect to the usual norm on $C_b(\Omega)$ the semigroup $\cT$ is bounded.
\end{proof}

In \cite{addona_angiuli_lorenzi2019}, the authors defined a \emph{system of invariant measures} for the semigroup $\bT$  as a family $\{\mu_k : k = 1, \ldots, m\}$ of positive finite Borel measures on $\R^d$ such that for every $\bff \in C_b(\R^d; \R^m)$,
\begin{equation}\label{eq.invariant}
	\sum_{k=1}^m \int_{\R^d} (\bT_t\bff)_k\, d\mu_k = \sum_{k=1}^m \int_{\R^d} f_k \, d\mu_k.
\end{equation}
By identifying the family $\{\mu_1, \ldots, \mu_m\}$ with the measure $\mu$ on $\Omega$ given by
\[
	\mu (\{k\}\times A ) = \mu_k (A),
\]
Equation~\eqref{eq.invariant} is equivalent to 
\[
	\int_\Omega T_t f \, d\mu = \int_\Omega f\, d\mu,
\]
which simply means that $\mu$ is an invariant measure for the semigroup $\cT'$.

The following result is important for the asymptotic behavior.

\begin{prop}\label{p.fixed}
	Under the assumption of Hypothesis~\ref{hyp2} the following holds.
	\begin{enumerate}[(a)]
		\item The space $\fix\cT$ is the span of the vector $\chi$ given by $\chi (k,x) \coloneqq \xi_k$ for every $(k,x) \in \Omega$ (where $\xi \in \R^m \setminus \{0\}$ is the vector from Hypothesis~\ref{hyp2}(b)(iii)).
		\item There exists a (unique up to strictly positive multiples) non-zero invariant measure $0 \le \mu \in \cM(\Omega)$.
		\item One has $\langle \mu, \chi \rangle \not= 0$.
	\end{enumerate}
\end{prop}

\begin{proof}
	Part~(a) is exactly \cite[Prop.\ 3.2]{addona_angiuli_lorenzi2019}, part~(b) follows from \cite[Thm.\ 3.5]{addona_angiuli_lorenzi2019}. Part~(c) follows from the special form of $\chi$ and of $\mu$, where the latter is given in \cite[Thm.\ 3.5]{addona_angiuli_lorenzi2019}.
\end{proof}

\begin{rem}
	In the proof of \cite[Thm.\ 3.5]{addona_angiuli_lorenzi2019} the authors verify directly
that if $\nu$ denotes an invariant measure for the semigroup generated by the scalar operator $A$, then $\xi\nu$ is a system of invariant measures for $\bT$.
	This is rather straightforward and, actually, a rather short part of the proof of \cite[Thm.\ 3.5]{addona_angiuli_lorenzi2019}. The rest of the proof is actually devoted to showing uniqueness, more precisely that any other system of invariant measures is a scalar multiple of this. However, as explained in Remark \ref{r.converse-separation}, uniqueness of the invariant measure follows directly from the fact that $\fix \mathscr{T}$ is one-dimensional.
\end{rem}

In \cite[Sect.\ 4]{addona_angiuli_lorenzi2019} the authors prove, under additional regularity and growth assumptions on the coefficients $a_{ij}$ and $b_j$, that the semigroup converges as $t\to \infty$. The proof of this in \cite{addona_angiuli_lorenzi2019} is rather involved, it is 5 pages long. By using the abstract results in Theorem~\ref{t.main-cont-functions} and Corollary~\ref{c.main}, one obtains the following result immediately from the properties of the semigroup established in Proposition \ref{p.fixed}, without any additional growth assumptions on the coefficients.

\begin{thm} \label{t.conv-for-irred-system}
	Assume Hypothesis~\ref{hyp2} and let $0 \le \mu \in \cM(\Omega)$ be the non-zero invariant measure for $\cT$, normalized such that $\int_\Omega \chi\, d\mu = 1$. Then:
	\begin{enumerate}[(a)]
		\item For each $f \in C_b(\Omega)$, $T_tf \to \int_\Omega f\, d\mu \cdot \chi$ uniformly on compact subsets of $\Omega$ as $t \to \infty$.
		\item For each $\nu \in \cM(\Omega)$, $T_t'\nu \to \int_\Omega \chi\, d\nu \cdot \mu$ in total variation norm as $t \to \infty$.
	\end{enumerate}
\end{thm}

\subsection{Reducible systems coupled by matrix potentials} \label{s.appl-red-systems}

Now consider what happens when, in the situation of the previous subsection, Hypotheses~\ref{hyp2}(b)(iii) and~(iv) are dropped. 
A glance at \cite{addona_angiuli_lorenzi2019} shows that everything that was said before Proposition~\ref{p.fixed} in the previous subsection remains true without these two assumptions, except that $\bigcap_{x \in \R^d} \ker \bC(x)$ might no longer be one-dimensional. The following observation will be important.

\begin{rem} \label{r.dual-kernel}
	One has
	\begin{align*}
		\ker \bC(x) = \ker \bC(x)^*
	\end{align*}
	for each $x \in \R^d$ in this situation. 
	This equality was stated in \cite[Lem.\ 2.2]{addona_angiuli_lorenzi2019} under the assumption in Hypothesis~\ref{hyp2}(b)(iii); however, the proof there shows that this equality is actually a mere consequence of the dissipativity estimate in Hypothesis~\ref{hyp2}(b)(i). For an alternative proof with a more operator theoretic flavour, see Proposition~\ref{p.fixed-space-of-contractions} in the appendix.
\end{rem}

The goal of this subsection is to show that the semigroup $\cT$ (equivalently, $\bT$), introduced in the previous subsection, still converges at $t \to \infty$, even without the assumptions in Hypothesis~\ref{hyp2}(b)(iii) and~(iv) (but note that the limit operator need no longer have rank $1$, of course). There are at least two possible courses of action in order to prove this:

\begin{enumerate}[(i)]
	\item 
	One can try to subdivide the problem into smaller systems, so that condition (b)(iv) in Hypothesis \ref{hyp2} is satisfied for these smaller systems. Making use of Remark \ref{r.dual-kernel}, one can then show that, either also condition (b)(iii) in Hypothesis \ref{hyp2} is satisfied for the subsystem, or, else, $\bigcap_{x\in \R^d}\ker \bC(x) = \{0\}$. In the latter case, $\fix \cT=\{0\}$ and  Theorem~\ref{t.main-cont-functions} implies convergence (to $0$) of $T_t$ as $t \to \infty$.
	
	\item 
	One shows that $\fix \cT' \subseteq \cM(\Omega)$ separates $\fix \cT$ and then immediately applies Theorem~\ref{t.main-cont-functions}.
\end{enumerate}

Option~(i) seems to be rather cumbersome, and we have not checked whether all details work out well. Instead, the details of approach~(ii) are explained next. Define 
\begin{align*}
	F := \bigcap_{x \in \R^d} \ker \bC(x) \subseteq \R^m.
\end{align*}
Following the arguments in Step~1 of the proof of~\cite[Prop.\ 3.2]{addona_angiuli_lorenzi2019}, one can show that
\begin{align*}
	\fix \bT = \{\xi \one_{\R^d}: \, \xi \in F\};
\end{align*}
this immediately yields a description of $\fix \cT$. On the other hand, repeating the arguments in Step~1 in the proof of \cite[Thm.\ 3.5]{addona_angiuli_lorenzi2019}, one sees that, for each $\xi \in F$, the measure $\mu \in \cM(\Omega)$, given by
\begin{align*}
	\langle \mu, f\rangle = \sum_{k=1}^m \xi_k \langle \mu_0, f(k,\argument) \rangle
\end{align*}
for each $f \in \cM(\Omega)$, is in $\fix \cT'$; here, $0 \le \mu_0 \in \cM(\R^d)$ is a non-zero invariant measure for the semigroup associated to the scalar operator~\eqref{eq.scalara}. (To see that Step~1 in the proof of \cite[Thm.\ 3.5]{addona_angiuli_lorenzi2019} does not require Hypothesis~\ref{hyp2}(b)(iv), one can use Remark~\ref{r.dual-kernel} above.)

This implies that $\fix \cT'$ separates $\fix \cT$, so Theorem~\ref{t.main-cont-functions} implies the following convergence result.

\begin{thm} \label{t.conv-for-red-system}
	Assume Hypothesis~\ref{hyp2} but skip assumptions~(b)(iii) and~(b)(iv). Then:
	\begin{enumerate}[(a)]
		\item For each $f \in C_b(\Omega)$, $T_tf$ converges uniformly on compact subsets of $\Omega$ to a function $Pf \in C_b(\Omega)$ as $t \to \infty$.
		
		\item For each $\nu \in \cM(\Omega)$, $T_t'\nu$ converges in total variation norm to a measure $Q\nu \in \cM(\Omega)$ as $t \to \infty$.
	\end{enumerate}
	The rank of the limit operator $P$ and $Q$ equals $\dim F$ and is thus not larger than $m$.
\end{thm}

\begin{rem}
	The assumption in Hypothesis~\ref{hyp2}(b)(ii) is crucial for the analysis in Subsections~\ref{s.appl-irred-systems} and~\ref{s.appl-red-systems} since it ensures that the semigroup $\bT$ (equivalently, $\cT$) is positive. Without this positivity assumption, the entire convergence theory for positive semigroups is no longer applicable.
	
	However, there is in alternative approach based on contractivity rather than positivity assumptions, at least if the coupling potential is bounded and the semigroup satisfies an additional compactness property. See \cite{dobrick_glueck_2021} for details. 
\end{rem}

\appendix

\section{A note a ultra Feller operators}

The purpose of this section is to prove Proposition~\ref{p.ultrafeller2} below which was mentioned in Section~\ref{s.strong-feller-semigroups}; it is not needed in the main text, but interesting in its own right.
The following lemma is needed for the proof.

\begin{lem}
	\label{l.bipolar}
	Let $\Omega$ be a Polish space. 
	Then the closed unit ball of $B_b(\Omega)$ is weak${}^*$-dense in the closed unit ball of the norm dual $\mathscr{M}(\Omega)^*$.
\end{lem}
\begin{proof}
	Denote the closed unit balls in $B_b(\Omega)$ and $\mathscr{M}(\Omega)^*$ by $B_{B_b}$ and $\mathscr{M}(\Omega)^*$, respectively.
	
	We consider the dual pair $(\mathscr{M}(\Omega)^*, \mathscr{M}(\Omega))$ and compute the bipolar of $B_{B_b}$. One has
	\[
		(B_{B_b})^\circ \coloneqq \{\mu \in \mathscr{M}(\Omega) : \langle f, \mu \rangle \le 1 \, \forall \, f \in B_{B_b}\} = 
B_{\mathscr{M}},
	\]
	where $B_{\mathscr{M}}$ denotes the closed unit ball of $\mathscr{M}(\Omega)$. In the last equality, the inclusion ``$\supseteq$'' is trivial, whereas the converse inclusion follows from the fact that $B_b(\Omega)$ is norming for $\mathscr{M}(\Omega)$, which follows from the definition of the total variation of a measure. Similarly, one can see that $(B_{\mathscr{M}})^\circ = B_{\mathscr{M}^*}$. Thus, $(B_{B_b})^{\circ\circ} = B_{\mathscr{M}}$. On the other hand, by the bipolar theorem, $(B_{B_b})^{\circ\circ}$ is the $\sigma(\mathscr{M}(\Omega)^*, \mathscr{M}(\Omega))$-closed convex hull of $B_{B_b}$ in $\mathscr{M}(\Omega)^*$. This proves that $B_{B_b}$ is indeed weak$^*$-dense in $B_{\mathscr{M}^*}$.
\end{proof}

\begin{prop}
	\label{p.ultrafeller2}
	Let $\Omega$ be a Polish space, let $T$ be a kernel operator on $B_b(\Omega)$ which is ultra Feller and set $S\coloneqq T'$. Then $S^*\mathscr{M}(\Omega)^* \subseteq C_b(\Omega)$.
\end{prop}
\begin{proof}
	First one observes that $S^*\mathscr{M}(\Omega)^* \subseteq C_b(\Omega)$. To see this, let $\varphi \in \mathscr{M}(\Omega)^*$ be given. By Lemma~\ref{l.bipolar} there exists a bounded net $(f_\alpha) \subseteq B_b(\Omega)$ such that $f_\alpha \rightharpoonup^* \varphi$. As $T$ is ultra Feller and $f_\alpha$ is bounded, $Tf_\alpha$ is relatively $\beta_0$-compact. Passing to a subnet, we may assume that $Tf_\alpha$ converges with respect to $\beta_0$ to a bounded and continuous function $g$. As this entails weak$^*$-convergence in $\mathscr{M}(\Omega)^*$ and since $S^*$ is weak$^*$-continuous (so that $S^*f_\alpha \rightharpoonup^* S^*\varphi$), it follows that $S^*\varphi = g \in C_b(\Omega)$.
\end{proof}

\section{A note on the kernel of dissipative matrices}

The following result is mentioned in Subsection~\ref{s.appl-red-systems}. For the convenience of the reader, we include its simple proof.

\begin{prop} \label{p.fixed-space-of-contractions}
	Endow $\R^{m \times m}$ with the operator norm induced by the Euclidean norm on $\R^m$.
	\begin{enumerate}[\upshape (a)]
		\item If $D \in \R^{m \times m}$ and $\|D\| \le 1$, then $\fix D = \fix D^*$.
		
		\item If $C \in \R^{m \times m}$ and $\langle Cy, y \rangle \le 0$ for all $y \in \R^d$, then $\ker C = \ker C^*$.
	\end{enumerate}
\end{prop}
\begin{proof}
	(a) Let $P \in \R^{d \times d}$ denote the mean ergodic projection associated to $D$. Then $\fix D = P\R^d$ and $\fix D^* = P^* \R^d$. As $D$ is contractive, so is $P$; hence, the projection $P$ is orthogonal, i.e., $P = P^*$.
	
	(b) The dissipativity estimate for $C$ implies that $\|e^{tC}\| \le 1$ for each $t \in [0,\infty)$. Hence, it follows from~(a) that $\fix (e^{tC}) = \fix (e^{tC^*})$ for each $t \in [0,\infty)$, which implies the assertion.
\end{proof}

The same result remains true in infinite dimensions \cite[Corollary~8.7]{eisner2015}.

\bibliographystyle{abbrv}
\bibliography{stability}

\end{document}